\documentclass{amsart}
\usepackage[utf8]{inputenc}
\usepackage{graphicx,graphics}
\usepackage{fullpage}

\usepackage{amsmath,amssymb}
\usepackage{cite}
\usepackage{hyperref}

\usepackage{epstopdf}
\usepackage{bbm}
\usepackage{scalerel}
\usepackage[inline]{enumitem}

\usepackage{mathtools}
\usepackage{color,xcolor} 
\usepackage{mathrsfs,bbm} 
\DeclareMathAlphabet\mathbfcal{OMS}{cmsy}{b}{n}

\numberwithin{equation}{section}

\newcommand{\norm}[1]{\left\lVert#1\right\rVert}
\newcommand{\abs}[1]{\left|#1\right|}

\newcommand{\cqg}{ g}

\newcommand{\cqK}{ K}

\newcommand{\cqW}{ W}

\newcommand{\R}{\mathbb{R}}

\newcommand{\pt}{\partial_t}
\newcommand{\bg}{g}

\newcommand{\K}{ K}
\newcommand{\X}{ X}
\newcommand{\Y}{ Y}

\newcommand{\Id}{\mathrm I\!\hspace{0.7pt}\mathrm d}

\newcommand{\pttau}{\partial_t^\tau}

\renewcommand{\H}{ H}
\renewcommand{\Re}{\text{Re }}
\renewcommand{\Im}{\text{Im }}
\newcommand{\epsm}{\epsilon_{\operatorname{mach}}}
\newtheorem{theorem}{Theorem}[section]
\newtheorem{lemma}[theorem]{Lemma}
\newtheorem{remark}[theorem]{Remark}
\newtheorem{corollary}[theorem]{Corollary}
\newcommand{\eremk}{\hbox{}\hfill\rule{0.8ex}{0.8ex}}
\newenvironment{numberedproof}[2][Proof]{\noindent \emph{#1 #2} }{\hfill \qed}

\author{Jens M. Melenk}
\address{ TU Wien, 
	Institute for Analysis and Scientific Computing, Wiedner Hauptrasse 8-10, 1040 Vienna, Austria}
\email{jens.melenk@tuwien.ac.at}
\author{ Jörg Nick}
\address{ETH Zürich,	Seminar for Applied Mathematics,
	 Rämistrasse 101, CH-8092 Zürich, Switzerland 
}
\email{joerg.nick@math.ethz.ch}
\title[Parsimonious convolution quadrature]{Parsimonious convolution quadrature}
\begin{document}
	\maketitle
\begin{abstract} 
We present a method to rapidly approximate convolution quadrature (CQ) approximations, based on a piecewise polynomial interpolation of the Laplace domain operator, which we call the \emph{parsimonious} convolution quadrature method. For implicit Euler and second order backward difference formula based discretizations, we require $O(\sqrt{N}\log N)$ evaluations in the Laplace domain to approximate $N$ time steps of the convolution quadrature method to satisfactory accuracy. The methodology proposed here differentiates from the well-understood fast and oblivious convolution quadrature \cite{SLL06}, since it is applicable to Laplace domain operator families that are only defined and polynomially bounded on a positive half space, which includes acoustic and electromagnetic wave scattering problems. The methods is applicable to linear and nonlinear integral equations. To elucidate the core idea, we give a complete and extensive analysis of the simplest case and derive worst-case estimates for the performance of parsimonious CQ based on the implicit Euler method.
For sectorial Laplace transforms, we obtain methods that require $O(\log^2 N)$ Laplace domain evaluations on the complex right-half space. 
We present different implementation strategies, which only differ slightly from the classical realization of CQ methods. Numerical experiments demonstrate the use of the method with a time-dependent acoustic scattering problem, which was discretized by the boundary element method in space.
\end{abstract}
\section{Introduction}
\label{sec:intro}
Convolution quadrature (CQ) methods are a collection of numerical schemes to approximate 
\begin{align*}
u(t) = \int_0^t k(t-t')g(t') \mathrm d t' \, , \quad 0 \le t\le T.
\end{align*}
An overview collecting the main results and applications of convolution quadrature methods is found in the recent book \cite{banjai-sayas22}. Both in the analysis and the numerical realization, the methodology does not rely on the time-domain kernel $k$, but rather on its Laplace transform $K = \mathcal L k$. In many applications, this Laplace domain function is more accessible, both for the theory and the numerics. The effectiveness of this approach has been demonstrated for a wide range of applications, 
in particular in the context of time-domain boundary integral equations for wave propagation problems \cite{BLS17,BL19,DES21,MR23,BF24}.

The simulation of waves posed on unbounded domain, e.g., in the context of acoustic or electromagnetic scattering, are naturally approached by time-domain boundary integral equations, which are therefore an attractive foundation for computational methods. 
Following the original work on convolution quadrature for hyperbolic problems in \cite{L94}, many contributions towards the acceleration of these methods have been made, such as computationally efficient formulations which decouple and reduce the number of necessary Laplace transform evaluations \cite{BS09,B10}, variable time stepping \cite{LS13,LS15,LS16}, and the combination with compression techniques \cite{HKS07,HKS09,BK14}.

In their original formulation, the convolution quadrature method requires $O(N)$ evaluations in the Laplace domain, where $N$ is the number of temporal samples or number of time steps. 
For sectorial Laplace transforms, the fast and oblivious convolution quadrature introduced in \cite{SLL06} significantly reduces the number of necessary Laplace transform evaluations to $O(\log N)$. For Laplace transforms that are only analytic on a positive half-space, which arise for example in the context of acoustic and electromagnetic scattering problems, such a technique remained elusive over the last two decades. For a class of dissipative wave equations, techniques based on oblivious quadrature have been developed in \cite{BLS17}, although generally still a linear dependence on the number of timesteps $N$ used and the number of necessary evaluations in the Laplace domain remains.
In many settings, in particular in the context of retarded boundary integral equations, 
the evaluation of the kernel in the Laplace domain is the computational bottleneck. This is the starting point of the present work. 
\subsection*{Outline and contributions}
We give an overview of the paper. The next section introduces the mathematical setting of temporal convolution operators and their discretization, which is a general framework that includes acoustic and electromagnetic scattering problems. Section~\ref{sect:parsimonious} introduces the technique in a general setting and collects interpolation error estimates for analytic functions from the literature. 
In Section~\ref{sec:euler}, we construct a specific method for the case of the implicit Euler method, which provably preserves the convergence rate of the original convolution quadrature method. Section~\ref{sect:higher-order} describes the extension to higher-order methods and shows that the methodology applies to CQ based on backward differential formula of order $2$ with asymptotically the same properties.  The theoretical part closes with Section~\ref{sect:sectorial}, which covers the case where $K(s)$ can be extended to a sector in the left complex half plane. Section~\ref{sect:numerics} closes the paper with some considerations with respect to the efficient implementation, as well as numerical experiments.

In the present work we provide the first mathematical evidence that, for Laplace transforms that are analytic only in a half-space, the number of Laplace transform evaluations in CQ can be sublinear, i.e., $O(N^\delta)$ for some $\delta < 1$ depending on the method. In view of the typical case of $K(s)$ being an approximation to an operator (e.g., by the boundary element method), we refer to the application $\varphi \rightarrow K(s)\varphi$ as a matrix--vector product. For methods based on the implicit Euler, or the backward differential formula of order $p=2$, the asymptotic properties of the method used to compute the convolution quadrature approximation with $N$ time steps can then be summarized as follows. 
\begin{itemize}
	\item  $O(\sqrt{N}\log N)$ evaluations of $K(s)$ have to be computed. 
	\item  In order to compute the approximation of $u(t_n)$ at a single time $t_n$, we require  $O(\sqrt{N}\log N)$ matrix--vector products.
	\item  To compute $u(t_n)$ at all discrete times $t_n = \tau n $ for $n=1,\dots,N$, we require $O(\sqrt{N}\log^2 N)$ matrix--vector products.
	\item  In the case of marching-on-in-time schemes, $O(\sqrt{N}\log N)$ evaluations in the Laplace domain have to be computed and stored, and $O(\sqrt{N}\log^3 N)$ matrix--vector products have to be computed.
\end{itemize}
With these limitations in mind, we believe the method to be particularly useful in the context of data with low regularity (as investigated in, e.g., \cite{BS12}) and for waves with dispersive material laws (see \cite{DES21,NBL23}) that include some dissipation, as is the case for metamaterials based on plasmonic resonators.

\section{Mathematical background}
\label{sect:setting}

The basic notation and results introduced in this section present the setting of \cite{L94}. More background can be found in \cite[Chapter~2]{banjai-sayas22}.
We start by introducing the setting of time-dependent convolutional operators and temporal Sobolev spaces.
Let $K(s)\colon X \rightarrow Y$ be an analytic family of bounded linear operators, defined for all $s$ in a complex positive half space
$\Re s > \sigma_0> 0$. The spaces $X$ and $Y$ denote arbitrary complex-valued Banach spaces.

The analytic operator family $K(s)$ is assumed to be polynomially bounded (with respect to the frequency $s$) for all $\Re s \ge 0$ in the following sense. 
There exist $\mu \in {\mathbb R}$ and  $M_0<\infty$ such that
\begin{align}\label{Z-bound}
\norm{K(s)}_{Y\leftarrow X}&\leq M_0 \abs{s}^\mu\quad  \text{ for \ Re } s \ge 0.
\end{align}
We note that this setting is a simplified version of the typical assumptions (see, e.g., \cite[Sec.~2.1]{L94}). This simplification is made for the sake of presentation, but does not restrict the proposed methodology. Implications of this simplification are discussed in Remark~\ref{rem:simpl-K} below. 
Under this condition, $\K(s)$ is guaranteed to be the Laplace transform of a distribution of finite order of differentiation with support on the nonnegative real half-line $t \ge 0$. 
Hence, for a sufficiently regular function $\bg:\mathbb R\to \X$ that vanishes on the negative real half-axis, we use the Heaviside notation of operational calculus, 
which writes 
\begin{equation}\label{Heaviside}
\K(\pt)\bg := (\mathcal{L}^{-1}\K) * \bg
\end{equation}
for the convolution of the inverse Laplace transform of $\K(s)$ with $\bg$. 
This notation defines a wide class of temporal differential operators and is motivated by the fact that for $\Id(s)=s$, we have $\Id(\pt)g=\pt g$, which is the time derivative of $g$.


Operators defined by the notation of operational calculus require a temporal functional analytic framework, which is provided by temporal Sobolev spaces.
Consider the space $\H^r(\R,\X)$ with real order $r$, the Sobolev space of order $r$ of $\X$-valued functions on $\R$, 
\cite[Chap.~1]{lions-magenes-I}.
Restricting the Hilbert space to functions that vanish on the negative real axis yields the definition (details on the definition are found in, e.g., \cite[Section~2.1]{banjai-sayas22})
$$
\H_0^r(0,T;\X) := \{\bg|_{(0,T)} \,:\, \bg \in \H^r(\R,\X)\ \text{ with }\ \bg = 0 \ \text{ on }\ (-\infty,0)\} . 
$$

The following identity from \cite[Lemma~{2.1}]{L94} shows that the temporal convolutions defined by the Heaviside notation extends to bounded operators on appropriate temporal Sobolev spaces. 
Let the analytic family $\K(s)$ satisfy the polynomial bound \eqref{Z-bound} in the half-plane $\text{Re }s > 0$.
Under this condition, $K(\pt)$ (defined by \eqref{Heaviside} for regular functions $g$) extends by density to a bounded linear operator 
\begin{equation}\label{sobolev-bound}
	\K(\pt) : \H^{r+\mu}_0(0,T;\X) \to \H^r_0(0,T;\Y)
\end{equation}
for arbitrary real $r$. The framework of temporal Sobolev spaces therefore provides an appropriate setting for the Heaviside notation of operational calculus. 
\begin{remark}[On the assumptions on $K$]\label{rem:simpl-K}
In many settings, e.g.,  for acoustic or electromagnetic scattering problems, we are interested in $K(s)$ that only fulfill the bound \eqref{Z-bound} 
on half spaces $\Re s \ge \sigma>0$ for some arbitrary $\sigma>0$. A formulation that fits into the present setting \eqref{Z-bound} is then constructed by using the shifted operator $\widetilde K(s) = K(s+\sigma)$. Then, we find that $\widetilde K(\partial_t) g =e^{-\sigma t} K(\partial_t)g $. Using $\sigma=1/T$ with the present setting then generalizes the present theory this case and avoids an exponential dependence on the terminal time $T$. 
\eremk
\end{remark}

\subsection{The convolution quadrature method}
\label{ap:cq}
The convolution quadrature method gives, based on an underlying A-stable time stepping scheme with step size $\tau>0$, 
an approximation of the temporal convolutions $\K(\partial_t)$ on the time interval $[0,T] = [0, \tau N ]$. 
The time stepping scheme determines the so-called generating function $\delta$, which can be expanded in a power series
\begin{equation*}
\delta(\zeta) = \sum_{j=0}^\infty \delta_j \zeta^j.
\end{equation*}
Details are found in the original works \cite{L88,L94} and the recent book \cite{banjai-sayas22}. The generating functions of the first order (implicit Euler method) 
or second backward difference formula (BDF2) are given by $\delta(\zeta)=(1-\zeta)$ and $\delta(\zeta)=(1-\zeta)+\frac{1}{2}(1-\zeta)^2$, respectively. 
For multistage methods these generating functions are matrix-valued. Although Runge--Kutta based methods typically outperform their multistep counterparts in the context of convolution quadrature methods (see, e.g., \cite{B10}), we will mostly discuss multistep based methods in this paper due to their simplicity.


For a time step $\tau > 0$, the convolution quadrature weights are defined as the coefficients of the power series 
\begin{equation}\label{rkcq-weights}
\K\Bigl(\frac{\delta(\zeta)}\tau \Bigr) = \sum_{n=0}^\infty {W}_n(\K) \zeta^n.
\end{equation}
The operators ${W}_n(\K):\X \to Y$ are the quadrature weights of the convolution, namely, for $g=(g^n)_{n\in {\mathbb N}_0}$, we say that $K(\partial_t^\tau)g$ is the convolution quadrature approximation to the temporal convolution \eqref{Heaviside} (with an appropriate right-hand side), which is defined by the discrete convolution
\begin{equation}\label{rkcq}
\bigl(\K(\pttau) g \bigr)_n = \sum_{j=0}^n {W}_{n-j}(\K) g^j.
\end{equation}
This procedure thus leads to the approximation 
\begin{equation*}
\bigl(\K(\pttau) g \bigr)_n  \approx 
\bigl(\K(\pt) g \bigr)(n \tau), \qquad n=0,\ldots, N, 
\end{equation*}
at the discrete time points $t_n = n \tau$,  where the sequence $(g^n)_{n \in {\mathbb N}_0} $ is obtained by sampling $g$: 
$g^n = g(t_n)$. 

\subsection{Standard assembly of the quadrature weights}
To approximate the quadrature weights, we use the Cauchy integral formula, which reads for $\lambda <1$ 
\begin{align}
\label{eq:CQ-weights}
\cqW_n(\cqK) = \frac{1}{2\pi i} \int_{\abs{\zeta} = \lambda} \cqK(\delta(\zeta)/\tau)\zeta^{-n-1} \mathrm d \zeta
= \dfrac{1}{2\pi\lambda^{n}}\int_0^{2\pi}e^{-i n\varphi}\cqK(\delta(\lambda e^{i\varphi})/\tau)) \,\mathrm d \varphi.
\end{align}
The established method to approximate these integrals effectively is to employ the trapezoidal rule with $L\ge N$: 
\begin{align}
\cqW_n(\cqK)&\approx \cqW^\lambda_n(\cqK):= \dfrac{\lambda^{-n}}{L}\sum_{l=0}^{L-1}\cqK\left(\dfrac{ \delta(\lambda\, \zeta_L^{-l})}{\tau}\right)\zeta_L^{nl}, \quad \text{for } 0\le n \le N,
\end{align}
where $\zeta_L=  e^{2\pi i  /L}$.
Inserting these approximations of the quadrature weights into \eqref{rkcq} and using that $\cqW_j(\cqK) = 0$ for $j < 0$ gives the standard convolution quadature scheme
\begin{align}\label{eq:forward-cq}
\begin{aligned}
\left(\cqK(\pt^\tau)\cqg\right)_n  = 
\sum_{j=0}^N \cqW_{n-j}(\cqK) g^j 
&\approx
\sum_{j=0}^N \cqW^\lambda_{n-j}(\cqK) g^j 
\\&= 
\dfrac{\lambda^{-n}}{L}\sum_{l=0}^{L-1}\zeta_L^{ln} \cqK\left(\dfrac{\delta(\lambda \,\zeta^{-l}_{L})}{\tau} \right) \left[ \sum_{j=0}^N \lambda^j \cqg^j \zeta^{-jl}_{L} \right] =:\left(\cqK(\pt^{\tau,\lambda})\cqg)\right)_n. 
\end{aligned}
\end{align}
Here, the parameter $\lambda \in (0,1)$ should be chosen to balance the effect of various errors, in particular 
\begin{enumerate*}[label=(\alph*)]
\item 
the effects of finite precision arithmetic;  
\item
the quadrature error of the trapezoidal rule; 
\item 
errors introduced by further approximating $K(\delta(\lambda\zeta^{-l}_L)/{\tau})$ (see  
(\ref{eq:forward-cq-interpolation}) for the final form of the proposed parsimonious CQ).
\end{enumerate*}
The first two issues lead to a rather standard choice (see \cite{L94}) of the parameter $\lambda$ as 
$\lambda = \epsm^{1/(2N)},$
where $\epsm >0$ 
the machine precision. For kernels $K$ satisfying (\ref{Z-bound}), this choice ensures that roundoff errors are 
not larger than $O(\sqrt{\epsm})$, see, e.g., \cite[Sec.~{3.4}]{banjai-sayas22}. Furthermore, the use of the exponentially convergent 
trapezoidal rule entails an additional error (see, e.g., \cite[Lemma~{3.2}]{banjai-sayas22}) of size $O(\lambda^{N})$ for the choice 
$L = N$, thus again of size $O(\sqrt{\epsm}))$. We mention in passing that the achievable error can be lowered by larger choices of $L$ as discussed in 
\cite[Sec.~{3.4}]{banjai-sayas22}. The third issues arises from a kernel approximation that we propose in the present work and is discussed in the following 
Section~\ref{sec:kernel-approximation}. It will lead us to selecting $\lambda = \epsilon^{1/(2N)}$ with $\epsilon \ge \epsm$ 
such that the convergence rate of the standard convolution quadrature is retained (see  (\ref{eq:parameter-choice}) ahead).

\section{Parsimonious CQ by piecewise polynomial kernel approximation} \label{sect:parsimonious}
\subsection{Kernel approximation by piecewise polynomials}
\label{sec:kernel-approximation}
The aim of this work is to derive schemes that require only a sublinear (with respect to $N$) number of evaluations of the Laplace transform. 
To that end,  we replace the time-harmonic evaluations of $K$ in \eqref{eq:forward-cq} by a piecewise Chebyshev interpolant. On 
the reference interval $[-1,1]$, the Chebyshev interpolation operator $I^{[-1,1]}_p: C([-1,1]) \rightarrow {\mathcal P}_p$, mapping into the space 
${\mathcal P}_p$ of polynomials of degree $p$ is given by interpolating in the $p+1$ Chebyshev points 
$\{x_j^{(p)} = \cos (\frac{2j+1}{2p+2} \pi)\,|\, j=0,\ldots,p\}$. For a general interval $[a,b] \subset {\mathbb R}$, the 
Chebyshev interpolation operator $I^{[a,b]}_p: C([a,b]) \rightarrow {\mathcal P}_p$ is derived from $I^{[-1,1]}_p$ by affinely transforming to $[-1,1]$.  

For $\Phi_{-1}:= 0 < \Phi_0 < \Phi_1 < \cdots < \Phi_M := \pi$, we set $I_j:= [\Phi_{j-1},\Phi_j]$ 
and $I_{j+M+1}:= 2\pi - I_{M-j}$ for $j=0=1,\ldots,M$. 
On $[0,2\pi]$ we introduce the mesh ${\mathcal T}:= \{I_j\,|\, j=0,\ldots,2M+1\}$, which is symmetric with respect to $\pi$. 
We let the operator
$I^{\mathcal T}_p$ be defined as the elementwise degree-$p$ Chebyshev interpolation operator for $j=1,\ldots,2 M$ and as the identity operator 
$I^\Phi_p = \operatorname{I}$ for $j \in\{0,2M+1\}$, i.e., 
\begin{equation*}
(I^{\mathcal T}_p f)|_{I_0} = f, 
\qquad 
(I^{\mathcal T}_p f)|_{I_{2M+1}} = f, 
\qquad 
(I^{\mathcal T}_p f)|_{I_j} = I^{I_j}_p (f|_{I_j}), \quad j =1,\ldots,2M. 
\end{equation*}
This yields our final approximation, the parsimonious CQ:  
\begin{align}\label{eq:forward-cq-interpolation}
\left(\cqK(\pt^{\tau,\lambda})\cqg\right)_n \approx \left(I^{\mathcal T}_p \cqK(\pt^{\tau,\lambda})\cqg\right)_n := 
\dfrac{\lambda^{-n}}{L}\sum_{l=0}^{L-1}\zeta_L^{ln}  I^{\mathcal T}_{p}\cqK\left(\dfrac{\delta(\lambda \,\zeta^{-l}_{L})}{\tau} \right) \left[ \sum_{j=0}^N \lambda^j \cqg^j \zeta^{-jl}_{L} \right].
\end{align}
\begin{remark} 
The mesh ${\mathcal T}$ is symmetric with respect to the point $\pi$. This is not essential for the algorithm and its analysis and merely done to exploit 
algorithmically potential symmetries of $K$: If $K(\overline{s}) = \overline{K(s)}$, as is often the case, then (see \cite[Section~4.1]{BS09} for more details)
\begin{align*}
K\left( \dfrac{\delta( \lambda e^{-i\varphi }) }{\tau}\right)
	=K\left( \dfrac{\overline{\delta( \lambda e^{i\varphi }) }}{\tau}\right)	=	\overline{ K\left( \dfrac{\delta( \lambda e^{i\varphi }) }{\tau}\right)}, 
\end{align*}
which allows one to halve the number of applications of the operator $K$ in  
(\ref{eq:forward-cq-interpolation}). 
\eremk
\end{remark}
Restricted to an interval, the interpolation operator reads
\begin{align*}
  I^{\mathcal T}_p K\left( \dfrac{\delta( \lambda e^{i\varphi }) }{\tau}\right)\Bigg|_{[\Phi_{j-1},\Phi_j ]}
  = \sum_{j=0}^p
  T_j(\varphi) C_j .
\end{align*}
 The operator-valued Chebyshev coefficients $C_j$ are uniquely determined by evaluations of the left-hand side at the zeros of the $p$-th order Chebyshev polynomial transformed to the interval $[\Phi_{j-1},\Phi_j ]$.
\begin{remark}
The approach described here resembles Filon quadrature, which would replace the Laplace domain operator $K$ by its piecewise Chebyshev interpolation $I^{\mathcal T}_p$ in \eqref{eq:CQ-weights} and then compute the integral exactly. Here, we purposefully avoid this step and leave the numerical integration of the composite trapezoidal rule in place in order to preserve the structure of \eqref{eq:forward-cq}, where both summations can be implemented via the fast fourier transform (FFT). By relying on these techniques, we are able to give an efficient method to compute the approximation \eqref{eq:compute-all-times} at all times at once with $O(Mp^2)$ evaluations. If only a single time $t$ is of interest, then a Filon-type quadrature might be an interesting alternative. Details on the mathematical analysis and implementation of Filon quadrature are found in \cite{IS05} and an extensive number of subsequent papers.
\eremk
\end{remark}
\subsection{The domain of analyticity in the half space case} 
In order to estimate the interpolation error, we need to understand the domain of analyticity of the analytic extension of the scalar function 
\begin{align}\label{def:f - varphi}
    f(\varphi) = K\left( \dfrac{\delta( \lambda e^{i\varphi }) }{\tau}\right) \, \colon \, D^{\delta} \rightarrow \mathbb C .
\end{align}
More generally, the above expression describes an analytic family of operators, since $K(s)\,\colon \, X\rightarrow Y$. For the readability of the next sections, we restrict our attention to scalar transfer functions, i.e.,  $K(s)\in \mathbb C$. The generalization to general $K(s):X\rightarrow Y$ for Hilbert spaces $X$ and $Y$ is straightforward.

Let $D^\delta\subset \mathbb C$ be the domain of analyticity of $f$, i.e., all $\varphi \in \mathbb C$ such that the argument 
$\delta(\lambda e^{i \varphi})/\tau$ of the Laplace domain operator $K$ has positive real part:
\begin{align}\label{eq:cond-pos-Real}
	D^\delta := \{\varphi \in \mathbb C\colon \Re \delta(\lambda e^{i\varphi }) > 0\}.
\end{align}
Crucially, we are interested in the size of the largest Bernstein ellipse around the intervals $[\Phi_{j-1},\Phi_j ]$ that is contained in the domain of analyticity. 
We define 
\begin{align}
\label{eq:Drho}
D_\rho(-1,1) &:= \left\{z \in \mathbb C\,\colon \, |z - 1 | + |z +1| < \rho + 1/\rho\right\},  \\
\label{eq:Erho}
E_{\rho}(-1,1)&:= \partial D_\rho(-1,1) = 
	 \left\{z = \dfrac{\rho e^{i\theta} + \rho^{-1}e^{-i\theta}}{2} \, \colon \, \theta\in [0,2\pi] \right\}. 
\end{align}
Generalizing this parametrization to arbitrary intervals $[\phi_0, \phi_1]$ by the linear transformation $\psi \, \colon \, [-1,1] \rightarrow [\phi_0, \phi_1]$, 
$x \mapsto \phi_0 +(\phi_1 -\phi_0)\frac{x+1}{2}$ yields 
\begin{align}\label{def:Bernstein-general-I}
E_\rho(\phi_0,\phi_1) 
:=
\left\{
z = \phi_0 +\frac{1}{2}(\phi_1 - \phi_0)\left(\dfrac{\rho e^{i\theta} + \rho^{-1}e^{-i\theta}}{2} +1\right) \, \colon \, \theta\in [0,2\pi]
\right\} .
\end{align}
With these ellipses, we are in position to formulate a polynomial approximation result for Chebyshev interpolation. 
%
\begin{lemma}
\label{lemma:polynomial-approximation}
Let $f\,:\, D_\rho(-1,1) \rightarrow \mathbb C $ be holomorphic.  
The error of the Chebyshev interpolant $I^{[-1,1]}_p f \in {\mathcal P}_p$ on the interval $[-1,1]$ is bounded by 
\begin{align*}
	\norm{f-I^{[-1,1]}_p f}_{L^{\infty}(-1,1)} \le 2^{3/2}\sqrt{1-\rho^{-2}} \rho^{-p}\norm{f}_{L^{\infty}(D_\rho(-1,1))}.
\end{align*}
 If $f|_{(-1,1)}$ is real-valued, then $I_p f$ is likewise real-valued on $(-1,1)$. 
\end{lemma}
\begin{proof}
See \cite[Lemma~{7.3.3}]{sauter-schwab11}. An alternative proof (with a possibly different constant) can be inferred from the explicit relation
between the expansion coefficients of the Chebyshev expansion of $f$ and the coefficients of the interpolant \cite[Thm.~{4.2}]{trefethen13}. 
\end{proof}
Via pull-back, this result generalizes to arbitrary intervals: 
\begin{equation}
\label{eq:chebyshev-interpolation-error-scaled}
\|f - I^{I_j}_p f\|_{L^\infty(I_j)} \leq 2^{3/2} \sqrt{1-\rho^{-2}} \rho^{-p} \norm{f}_{L^\infty(E_\rho(\Phi_{j-1},\Phi_j))}, 
\qquad j=1,\ldots,2 M. 
\end{equation}
\section{Parsimonious CQ based on the implicit Euler method}
\label{sec:euler}
\subsection{A specific choice of mesh for the implicit Euler method}
\label{sec:implicit-euler}
In the following, we describe an explicit collection of intervals, for which the piecewise polynomial interpolation for \eqref{def:f - varphi} converges exponentially with an explicit, predetermined convergence rate. In order to simplify the expressions, we use the rate $\rho=2$ as an example, although the techniques generalize in a straightforward way to arbitrary $\rho>1$.
%
Define $\Phi_j$, $j \in \mathbb N_0$,  recursively by 
\begin{subequations}
\label{eq:rec}
\begin{align}
\label{eq:rec-a}
\Phi_0 &:= \sqrt{\frac{10| \log (\epsilon)|}{6N}}, \\
\label{eq:rec-b}
	\Phi_{j+1} 
	&= 
\begin{cases} \Phi_j + 
\min \left\{\Phi_j, -\frac{8}{3} \log(\lambda) - \frac{8}{3} \log\left(\cos\left(\frac{7}{8} \Phi_j\right)\right) \right\} \quad &\text{ if } \Phi_j < 1 \\
\pi & \text{else} .
\end{cases} 
\end{align}
\end{subequations}
The recursion terminates at the index
\begin{align}
\label{eq:M}
	M= \min \left\{j\in \mathbb N \, | \, \Phi_j = \pi \right\}.
\end{align}
We note that the sequence is strictly increasing, i.e. $\Phi_j<\Phi_{j+1}$, for $j\le M-1$. As described in Section~\ref{sec:kernel-approximation},
we consider the mesh ${\mathcal T}$ with $2M+2$ elements determined by the $\Phi_j$, $j=0,\ldots,M$. 
The intervals $[\Phi_j,\Phi_{j+1}]$ are constructed such that an efficient convolution quadrature scheme based on piecewise polynomial approximation of $f$
is possible: 
\begin{theorem}[Intervals for the implicit Euler method]
\label{thm:euler}
Let $K$ satisfy (\ref{Z-bound}). Let the mesh ${\mathcal T} = \{[\Phi_{j-1},\Phi_j]\,|\, j=0,\ldots,2M+1\}$ be given by 
(\ref{eq:rec}).  
Select $0 < \epsilon \ll 1$ and assume 
$0.08 \ge -\log \lambda  = -\frac{\log \epsilon}{2N}>0$ and $\log N\le |\log \epsilon |^{3/2} $ (both of these assumptions are fulfilled by all reasonable parameter sets of the convolution quadrature method for sufficiently large $N$).

Then: 
\begin{enumerate}[label=(\roman*),leftmargin=18pt]
\item 
\label{item:thm:euler-i}
The Chebyshev interpolation of $K$ on any interval $[\Phi_j,\Phi_{j+1}]$ converges exponentially with the rate $\rho=2$, more precisely, 
\begin{align}
\label{eq:thm:euler-1}
\max_{j=0,\ldots,2M+1} 
\sup_{s \in [\Phi_{j-1},\Phi_j]} \left\|K(s) - (I^{[\Phi_{j-1},\Phi_j]}_p K)(s)  \right\|_{Y \leftarrow \X} \le \frac{C M_{0} }{\tau^{\mu}2^p}.
\end{align}
The constant $C$ is independent of $f$, $p$, $\tau$, $T$, and $\epsilon$. 
\item 
\label{item:thm:euler-ii}
The number $M$ of intervals in \eqref{eq:M} satisfies 
\begin{align}
	\label{eq:thm:euler-2}
	M \leq 3 + 3\sqrt{N}. 
\end{align}
\end{enumerate}
\end{theorem}
\begin{proof}
\emph{Proof of \ref{item:thm:euler-i}:}
Recall that $\delta(z) = 1-z$. 
By symmetry of the mesh ${\mathcal T}$ with respect to $\pi$, it suffices to show (\ref{eq:thm:euler-1}) 
for $j=1,\ldots,M$. Note that the case $j = 0$ is trivial as $I^{I_0}_p$ is the identity operator.  
The key step is to check that the construction of the mesh ${\mathcal T}$ is such that 
\begin{equation}	
\label{eq:key-step-thm:euler}
\delta(\lambda( E_\rho(\Phi_j,\Phi_{j+1}))) \subset \mathbb C_+ = \{s\in \mathbb C \, | \, \Re s>0 \}. 
\end{equation}
Details of this calculation are relegated to Appendix~\ref{sec:details-of-thm:euler}. 
Accepting (\ref{eq:key-step-thm:euler}), we get 
for the ellipse $D_\rho(\Phi_j,\Phi_{j+1}) \subset \mathbb{C}$ enclosed by $E_\rho(\Phi_{j},\Phi_{j+1})$, 
the assumption (\ref{Z-bound})
\begin{align*}
\sup_{\Phi \in D_\rho(\Phi_j,\Phi_{j+1})} \left| K\left( \dfrac{\delta( \lambda e^{i\Phi }) }{\tau}\right) \right| \le \frac{CM_0}{\tau^\mu},
\end{align*}
where $C$ is a constant independent of $K$, $\lambda$, $\tau$ and $\Phi$. The approximation result \eqref{eq:thm:euler-1} follows from 
combining the best approximation result Lemma~\ref{lemma:polynomial-approximation} as formulated 
in (\ref{eq:chebyshev-interpolation-error-scaled}). 

\emph{Proof of \ref{item:thm:euler-ii}:} For $\Phi_{j-1}\le 1/2$, one ascertain (see Appendix~\ref{sec:details-of-thm:euler})
	\begin{align}\label{factor-Phi-vorn}
		\Phi_{j} \stackrel{\eqref{eq:thm:euler-10}}{\ge} \Phi_{j-1} - \frac{8}{3} \log \cos (\frac{7}{8} \Phi_{j-1}) 
\stackrel{\eqref{eq:thm:euler-20}}{\ge} \Phi_{j-1} + \frac{8}{6} \frac{49}{64} \Phi_{j-1}^2
\ge \left( 1+ \frac{49}{48} \Phi_{j-1}\right)\Phi_{j-1} \ge \cdots 
\ge \left( 1+ \frac{49}{48} \Phi_{0}\right)^j\Phi_{0}. 
	\end{align}
Let $j_\ast^{\epsilon}:= \min\{j\, \colon \, \Phi_j >  |\log \epsilon | \Phi_0 \}$. We infer, with $\log (1+ x) \ge \frac{1}{2} x$ (valid for $0 \leq x \leq 1$ and noting $\Phi_0 \leq 0.52$)
\begin{align*}
	j_\ast^{\epsilon} \leq \frac{\log|\log \epsilon|}{\log (1 + \frac{49}{48}\Phi_0)}
	\leq 2\frac{48}{49}\frac{\log |\log \epsilon|}{\Phi_0} 
	=2\frac{48}{49}\sqrt{\dfrac{6}{10}N} 
	\dfrac{\log |\log \epsilon|}{\sqrt{|\log \epsilon|}} \le 2 \sqrt{\dfrac{6}{10}N}, 
\end{align*}
where the last inequality holds for $\epsilon\le 1/3$ from a standard curve discussion. If $|\log \epsilon | \Phi_0<1/2$, we iterate this process once, observing
\begin{align*}
	\Phi_{j+j_*^{\epsilon}}  
	\ge \left( 1+ \frac{49}{48} \Phi_{j^{\epsilon}_*}\right)^j\Phi_{j^{\epsilon}_*}
	>
	\left( 1+ \frac{49}{48} |\log \epsilon |\Phi_{0}\right)^j|\log \epsilon|\Phi_0 . 
\end{align*}
In this case, we bound $j_\ast:= \min\{j\, \colon \, \Phi_{j+j_*^\epsilon} >  1/2\}$ via
\begin{align*}
	j_* \le \frac{-\log (2 |\log \epsilon |\Phi_{0})}{\log (1 + \frac{49}{48} |\log \epsilon |\Phi_{0})}
	\le 
	-2\frac{48}{49}\frac{\log (2 |\log \epsilon |\Phi_{0})}{|\log \epsilon |\Phi_{0} }
	\le 
-	2\frac{48}{49}\log \left(2 \sqrt{\dfrac{10}{6} \dfrac{|\log\epsilon|^3}{N}}\right)
	\sqrt{\dfrac{6}{10}\dfrac{N}{|\log \epsilon |^3}} 
. 
\end{align*}

Noting that $\log \cos \frac{7}{8} \frac{1}{2} \approx -0.26$, it is easy to see that at most $3$ elements are contained in $[1/2,\pi]$. 
Hence, we obtain the estimate
\begin{align*}
	M \leq 3 + j^\ast_\epsilon + j^\ast  
	&\leq
	 3 +  
	 \sqrt{\dfrac{6}{10}N} \left( 2  
-	2\frac{48}{49}\log \left( 2 \sqrt{\dfrac{40}{6} \dfrac{|\log\epsilon|^3}{N}}\right)
\dfrac{1}{|\log \epsilon |^{3/2}} 
\right)
\\	&\leq
3 +  
\sqrt{\dfrac{6}{10}N} \left( 2 + 
\frac{48}{49}\log \left(  \dfrac{6}{40} \dfrac{N}{|\log\epsilon|^3}\right)
\dfrac{1}{|\log \epsilon |^{3/2}} 
\right)
\\ & \le
	3+ \sqrt{\dfrac{6}{10}N} \left( 2 
+	\frac{48}{49}
\dfrac{\log \left(N\right)}{|\log \epsilon |^{3/2}} 
-
	\frac{48}{49}
\dfrac{\log \left(40/6 |\log \epsilon |^{3/2}\right)}{|\log \epsilon |^{3/2}} 
\right)
\le 
3 + 3\sqrt{N}
	,
\end{align*}
where we used the assumed $\log N\le |\log \epsilon |^{3/2} $ in the last estimate.
\end{proof}
\begin{remark}
\label{rem:motivation-recursion-euler}
The recursion \eqref{eq:rec} is an example of a partition of $[0,\pi]$ that ensures (\ref{eq:key-step-thm:euler}). Generalizations 
can be motivated as follows. Upon writing a complex number as $\varphi = \varphi_r + i \varphi_i$, $\varphi_r$, $\varphi_i \in {\mathbb R}$, 
the condition  that $\Re \delta(\lambda e^{i \varphi}) \ge 0 $ is, for the implicit Euler method with $\delta(\zeta) = 1 - \zeta$, equivalent to 
$1 - \lambda \exp(-\varphi_i) \cos \varphi_r \ge 0$ and thus (for small $\varphi_r$)
\begin{align}
\label{eq:domain-of-analyticity-euler}
\varphi_i \ge \log \lambda + \log \cos \varphi_r \approx \log \lambda - \tfrac{1}{2} \varphi_r^2. 
\end{align}
Depending on which of the two formulas is employed, the ratio of the semiaxes of a Bernstein ellipse for an interval $[\Phi_j, \Phi_{j+1}] \subset {\mathbb R}$ 
are therefore determined by the ratio 
\begin{equation}
\label{eq:ellipse-ratio}
R_1:= \frac{|\log \lambda + \log \cos \Phi_i|}{\Phi_{j+1} - \Phi_j} \approx 
\frac{|\log \lambda - \frac{1}{2} \Phi_i^2|}{\Phi_{j+1} - \Phi_j} =:R_2. 
\end{equation}
Requiring the ratios $R_1$ or $R_2$ to be constant (uniformly in $j$) leads to roughly uniform exponential rates of convergence for polynomial approximation, 
and we are led, depending on which of these two fractions $R_1$, $R_2$ is taken to be constant, to the recursions
\begin{align}
\label{eq:recursion}
\Phi_{j+1} = \Phi_j - c (\log \lambda + \log \cos \Phi_i) 
\qquad \mbox{ or } \qquad 
\Phi_{j+1} = \Phi_j - c (\log \lambda + \tfrac{1}{2} \Phi_j^2) 
\end{align}
for a fixed $c > 0$. 
The first recursion, which is based on $R_1$,  underlies the construction (\ref{eq:rec-b}), while the second one is slightly simpler. Let us focus on the second one. 
Note that $\log \lambda$ may be assumed to be small (cf.\ (\ref{eq:parameter-choice})).
A reasonable starting value $\Phi_0$ for the recursion is $\Phi_0 = O(\sqrt{|\log \lambda|})$. This
can be motivated by the observation that the mesh given by the recursion (\ref{eq:recursion}) changes character 
at the point $\overline{\Phi}:= \sqrt{2 |\log \lambda|}$: the elements $[\Phi_j,\Phi_{j+1}]$ to the left of $\overline{\Phi}$ are 
roughly equal in size and of size $O(|\log \lambda|)$, the elements to the right of $\overline{\Phi}$ grow quickly in size as $j$ increases. 
A more detailed analysis shows that not much is to be gained in terms of number of evaluations of the Laplace transform $K$ by including 
elements to the left of $\overline{\Phi}$. This suggests to select $\Phi_0 = O(\sqrt{|\log \lambda|})$. 
Finally, on the first interval $[0,\Phi_0] = [\Phi_{-1}, \Phi_0]$ polynomial interpolation appear inefficient as 
the ratio of the semiaxes of the Bernstein ellipse is $O(|\log \lambda|/\Phi_0) = O(\sqrt{|\log \lambda|})$. 
For the present case of small $\log \lambda$ the polynomial degree could not be chosen merely $O(\log N)$ for good accuracy. 
This motivates taking the operator $I^{\mathcal T}_p$ to be the identity on $[0,\Phi_0]$ and $[\Phi_{2M},\Phi_{2M+1}]$. 
\eremk
\end{remark}

\subsection{Approximation result for the implicit Euler method}\label{sect:approx-implic-euler}
With the approximation result Theorem~\ref{thm:euler} for the interpolation in the Laplace domain, we obtain the following temporal error bound.
\begin{theorem}[error of parsimonious CQ]
\label{thm:perturbation-argument}
	Let $K\,\colon \, \X \rightarrow \Y$ be an analytic family of operators satisfying (\ref{Z-bound}). Let 
the mesh ${\mathcal T}$ be given by (\ref{eq:rec}) and consider  (\ref{eq:forward-cq-interpolation}) based on the implicit Euler method. 
Then there is $C > 0$ independent of $N$, $L$ such that 
for any sequence $(g^n)_{n=1,\dots,N} \subset X$ 
	\begin{align*}
	\sqrt{ N^{-1} \sum_{n=0}^N\|	\left(\cqK(\pt^\tau)g-I^{\mathcal T}_p\cqK(\pt^{\tau,\lambda})\cqg\right)_n \big\|^2_Y}
	&\le C M_0 \tau^{-\mu} 
\left(  N 
\frac{\lambda^{L}}{1 - \lambda^{L}} + \lambda^{-N}  2^{-p}
\right) 
\sqrt{ N^{-1}\sum_{n=0}^N \|g^n\|^2_{\X} 
     }
     .
	\end{align*}
\end{theorem}
\begin{proof}
It is convenient to introduce for sequences $(x^j)_{j=0}^M \subset Z$ from a Hilbert space $Z$ the norm 
\begin{equation}
\label{eq:weighted-norm}
\|(x^j)_{j=0}^M\|_{\ell^2_M}:= \sqrt{ M^{-1} \sum_{j=0}^M \|x^j\|_Z^2 }. 
\end{equation}
\emph{Step 1:} From standard convolution quadrature analysis as given in \cite[Lemma~{3.1}]{banjai-sayas22} for the weights $\cqW_j(\cqK)$ 
and then \cite[Lemma~{3.2}]{banjai-sayas22}, which relies on the ``aliasing formula'', we have for $L\ge N$
\begin{align}
\label{eq:thm:perturbation-argument-10}
|\cqW_j(\cqK)| &\stackrel{\text{\cite[L.~{3.1}]{banjai-sayas22}}}{\lesssim} M_0 \tau^{-\mu}, \\
\label{eq:thm:perturbation-argument-20}
|\cqW_j(\cqK) - \cqW^L_j(\cqK)| &\stackrel{\text{\cite[L.~{3.2}]{banjai-sayas22}}}{\lesssim} M_0 \tau^{-\mu} \frac{\lambda^{L}}{1 - \lambda^{L}}. 
\end{align}
\emph{Step 2 (error estimate of classical realization of CQ):} 
\begin{align*}
	\left\|\left(\cqK(\pt^\tau)\cqg\right)_n -\left(\cqK(\pt^{\tau,\lambda})\cqg\right)_n \right\|^2_{\ell^2_N} & = 
N^{-1} \sum_{n=0}^N \left\|\sum_{j=0}^n (\cqW_{n-j}(\cqK) - \cqW^L_{n-j}(\cqK)) \cqg^j \right\|^2_\Y \\
& 
\stackrel{(\ref{eq:thm:perturbation-argument-20})}{\lesssim} N^2 \left(M_0 \tau^{-\mu}  \frac{\lambda^{L}}{1-\lambda^{L}}\right)^2 \|(g^j)_{j=0}^n\|^2_{\ell^2_N}. 
\end{align*}
\emph{Step 3 (perturbation error due to kernel interpolation):}
\begin{align*}
         &   N^{-1} \sum_{n=0}^N  \left(  \dfrac{\lambda^{-n}}{L} \bigg\| \sum_{l=0}^{L-
	1}\zeta_L^{ln} \bigg( \cqK\bigg(\dfrac{\delta(\lambda \,\zeta^{-
		l}_{L})}{\tau} \bigg)-I^{\mathcal T}_p \cqK\bigg(\dfrac{\delta(\lambda
	\,\zeta^{-l}_{L})}{\tau} \bigg) \bigg) \Big[ \sum_{j=0}^N \lambda^j
\cqg^j \zeta^{-jl}_{L} \Big] \bigg\|_Y\right)^2 
\\
& \stackrel{\text{Parseval}}{\lesssim}
 \dfrac{\lambda^{-2N}}{L^2} \sum_{l=0}^{L-1}  \left(  \bigg\| \bigg( \cqK\bigg(\dfrac{\delta(\lambda \,\zeta^{-
		l}_{L})}{\tau} \bigg)-I^{\mathcal T}_p \cqK\bigg(\dfrac{\delta(\lambda
	\,\zeta^{-l}_{L})}{\tau} \bigg) \bigg) \Big[ \sum_{j=0}^N \lambda^j
\cqg^j \zeta^{-jl}_{L} \Big] \bigg\|_Y\right)^2 
        \\ 
& \stackrel{\text{Thm.~\ref{thm:euler}}}{\lesssim} 
 \frac{\lambda^{-2N} }{L^2} \left(M_0 \tau^{-\mu} 2^{-p} \right)^2  \sum_{l=0}^{L-1} 
\left\|  \sum_{j=0}^N \lambda^j \cqg^j \zeta^{-jl}_{L} 
         \right\|^2_{\X}
        \\ 
& \stackrel{\text{Parseval}}{=}
\lambda^{-2N  }  \left(M_0 \tau^{-\mu} 2^{-p} \right)^2 
\frac{N}{L}  \|(\lambda^j g^j)_{j=0}^N\|^2_{\ell^2_N} 
\stackrel{\lambda < 1}{\leq  }
 \lambda^{-2N} \left(M_0 \tau^{-\mu} 2^{-p} \right)^2 
\|(g^j)_{j=0}^N\|^2_{\ell^2_N} , 
\end{align*}
where in the final estimate we used again $L\ge N$.
\end{proof}

\begin{remark}
The term $2^{-p}$ in Theorem~\ref{thm:perturbation-argument} is due to the use of the specific mesh ${\mathcal T}$ given by 
(\ref{eq:rec}). Other choices as discussed in Remark~\ref{rem:motivation-recursion-euler} also lead to exponential convergence
in $p$, albeit at a different rate. 
\eremk
\end{remark}
Theorem~\ref{thm:perturbation-argument} allows us to determine good choices of the parameters $\lambda$ and $p$ in dependence on the step size $\tau$. 
We assume $L = N$ for simplicity. 
First, we remark that Theorem~\ref{thm:perturbation-argument} ignores round-off errors due to finite precision arithmetic. 
Incorporating this and following the CQ literature on this issue (see, e.g., \cite[Sec.~{3.1}]{banjai-sayas22})
leads us to replace the term $2^{-p}$ with $\epsilon= \max\{2^{-p}, \epsm\}$. Next, 
$\lambda^{N}$ will be seen to be (asymptotically) small. Hence, estimating 
$N \lambda^N + \lambda^{-N} \epsilon \leq N (\lambda^N + \lambda^{-N} \epsilon)$ and then balancing 
the terms $\lambda^N$ and $\lambda^{-N} \epsilon$ yields 
\begin{equation*}
\lambda^{N} \stackrel{!}{=} \lambda^{-N} \epsilon= \max\{ 2^{-p},\epsm\}   \qquad \Longrightarrow \qquad \log \lambda = \frac{1}{2N} \log \epsilon. 
\end{equation*}
Next, assuming that the convolution quadrature has a convergence behavior 
\begin{equation}
\label{eq:assumed-convergence}
\|\cqK(\pt) g - \cqK(\pt^\tau) g\|_{\ell^2_N} \leq C \tau^k \|(g^n)_{n=0}^N\|_{\ell^2_N}
\end{equation}
for some $k > 0$, we are led to requiring 
\begin{align*}
\tau^k \stackrel{!}{=} N\tau^{-\mu} (\lambda^N + \lambda^{-N} \epsilon)  \approx 2 T \tau^{-\mu -1} \sqrt{\epsilon}. 
\end{align*}
Ignoring the factor $2 T$ leads us to 
\begin{align}
\label{eq:parameter-choice}
\epsilon := \max\{ \tau^{2(k+\mu+1)}, \epsm\} \qquad \mbox{ and } \qquad 
p = \lceil - \log_2 \epsilon \rceil
\qquad \mbox{ and } \qquad \lambda = \epsilon^{{1}/(2N)}.
\end{align}
\begin{corollary}[complexity estimates CQ based on implicit Euler method]
\label{cor:complexity-euler}
Consider (\ref{eq:forward-cq-interpolation}) based on the implicit Euler method and $L = N$. 
Assume that the CQ (\ref{rkcq}) converges such that (\ref{eq:assumed-convergence}) holds. Then the choice of parameters (\ref{eq:parameter-choice})
preserves the convergence rate of the CQ method. 
The size of the polynomial approximation mesh is $\#{\mathcal T} = O(\sqrt{N} )$ and the number of kernel evaluations is therefore $O(\sqrt{N}\sqrt{\log N}+M p )= O(\sqrt{N} \log N )$. 
\end{corollary}
\begin{proof}
Insert (\ref{eq:parameter-choice} in the complexity estimates of Theorem~\ref{thm:euler}. Note that 
on $[0,\Phi_0]$ the kernel is not approximated by a polynomial, which leads to $O(N \Phi_0) = O(N |\log \lambda|) = O( \sqrt{N} \sqrt{\log N})$ 
kernel evaluations in the first $[0,\Phi_0]$ and the last element $[\Phi_{2M}, 2\pi]$ of ${\mathcal T}$. 
\end{proof}
\begin{remark}[other meshes]
Remark~\ref{rem:motivation-recursion-euler} describes other meshes ${\mathcal T}$ based on different values $c$ and choices
of recursions in (\ref{eq:recursion}). This changes the exponential convergence behavior $O(2^{-p})$ to $O(e^{-b p})$ for some $b > 0$ 
in Theorems~\ref{thm:euler} and subsequently in Theorem~\ref{thm:perturbation-argument}. 
Relating $p = C |\log \epsilon|$ for suitable $C > 0$ in (\ref{eq:parameter-choice}) then still leads to the same convergence
and complexity estimates as given in Corollary~\ref{cor:complexity-euler}. 
\eremk
\end{remark}
\section{Parsimonious CQ based on higher order methods}\label{sect:higher-order}
\subsection{The case of BDF2} 
The procedure to design meshes ${\mathcal T}$ such that piecewise polynomial approximation of $\varphi \mapsto \cqK(\delta(\lambda e^{-i \varphi}/\tau))$ 
on $[0,2\pi]$ detailed in Section~\ref{sec:implicit-euler} can be generalized to higher order schemes. Let us illustrate the case of BDF2 using
the ideas developed in Remark~\ref{rem:motivation-recursion-euler}. 
Introduce the shorthand $\varphi = \varphi_r + i \varphi_i$, $\varphi_r$, $\varphi_i \in {\mathbb R}$. 
For the BDF2 method, we observe that \eqref{eq:cond-pos-Real} is equivalent to 
\begin{align*}
	0<
	\frac{3}{2}-2\lambda \cos(\varphi_r)e^{-\varphi_i}+\frac{1}{2}\lambda^2 \cos(2\varphi_r)e^{-2\varphi_i}.
\end{align*} 
For $\varphi_i=0$ the above condition is fulfilled due to the A-stability of the BDF2 method.
Rearranging for $\varphi_i$, applying Taylor series for the appearing logarithms and neglecting higher order terms in $\varphi_r$ yields the two zeros of the right-hand side

\begin{align*}
\varphi_1^i = \log(\lambda)-\frac{3}{2}\varphi_r^2 +O(\varphi_r^4), \quad
\varphi_2^i = \log(3)+O(\varphi_r^2) .
\end{align*}
The domain of analyticity is therefore given by the strip defined by 
\begin{align*}
	\log(\lambda)-\frac{3}{2}\varphi_r^2 +O(\varphi_r^4)
	\le
	\varphi_i
	\le
	\log(3)+O(\varphi_r^2) .
\end{align*}
For small $\varphi_r$, the left inequality is much stricter and is the significant restriction on the domain of analyticity. 
We find that the most critical part for the domain of analyticity is for $\varphi_r$ close to zero, where the restriction $\abs{\varphi_i}< \abs{\log(\lambda)}=  \tau \abs{\log(\sqrt{\epsilon})}$ has to be enforced. For small $\varphi_r$, we can use the approximation
\begin{align*}
	\dfrac{\abs{\log(\lambda)}+\frac{3}{2}\Phi_j^2}{\Phi_{j+1}-\Phi_j} = c ,
\end{align*}
where $c > 0$ is then a measure for the domain of analyticity of the scaled function, i.e., $c$ controls the rate of convergence of polynomial
interpolation on $[\Phi_j,\Phi_{j+1}]$. Enforcing a constant exponential convergence gives by rearranging the recursion as 
\begin{align}
\label{eq:rec-BDF2}
	\Phi_{j+1} = \Phi_j + c^{-1}\abs{\log(\lambda)} +\frac{3}{2c}\Phi_j^2 .
\end{align}
Structurally, the recursion is very similar to the one for the implicit Euler method in (\ref{eq:recursion}). 
In the same way as before for the case of the implicit Euler method discussed in 
Section~\ref{sec:euler} we observe analogous to Corollary~\ref{cor:complexity-euler} 
that $O (\sqrt{N}\log N)$ Laplace transform evaluations are sufficient to retain the convergence behavior of the CQ. 

\subsection{Higher order methods}
For higher order methods, in particular based on the $m$-stage Radau IIA Runge--Kutta method, the behavior of 
$\varphi \mapsto \delta(\lambda e^{- i\varphi}/\tau)$ for $\varphi$ near $\varphi = 0$ implies analogously to the procedure
in Remark~\ref{rem:motivation-recursion-euler} recursions of the form 
\begin{equation}
\label{eq:recursion-radau}
	\Phi_{j+1} = \Phi_j + c^{-1}\left(\abs{\log(\lambda)} +\xi\Phi_j^\kappa \right) .
\end{equation}
Here, $\xi>0$ and $\kappa\in \mathbb N$, $\kappa \ge 2$, are positive parameters determined by the Radau IIA time stepping scheme. 
As discussed in Remark~\ref{rem:motivation-recursion-euler}, a sensible choice of $\Phi_0$ is $\Phi_0 = O(|\log \lambda|^{1/\kappa})$ and,  
as argued in (\ref{eq:parameter-choice}), $| \log \lambda| = O(N^{-1} \log N)$. 
The mesh ${\mathcal T}$ given by (\ref{eq:recursion-radau}) 
supports exponential convergence in $p$ of the piecewise polynomial approximation. The number of Laplace transform evaluations of the parsimonious CQ
(\ref{eq:forward-cq-interpolation}) can be estimated as follows: 
\begin{align}\label{factor-phi-k}
	\Phi_{j} =(1+c^{-1}\xi\Phi^{\kappa-1}_{j-1}) \Phi_{j-1} > 
(1+c^{-1}\xi\Phi_0^{{\kappa-1} })\Phi_{j-1}
&> \left(1+c^{-1} \xi\Phi_0^{{\kappa-1} }\right)^{j-1} \Phi_0. 
\end{align}
 Repeating the argument structure of Part (ii) of Theorem~\ref{thm:euler} with $\Phi_{j^\epsilon_*} = O(|\log\epsilon|\Phi_0),$ whose index fulfills $$j^\epsilon_* = O(\log (|\log\epsilon |) / \log (1 + c^{-1} \xi \Phi_0^{\kappa-1}))$$ and subsequently $\Phi_{j_*} = O(1)$ for $j_* = O(|\log \Phi_0|/ \log (1 + c^{-1} \xi \Phi_0^{\kappa-1}))$ 
so that the number of elements of ${\mathcal T}$ is $M = O(\Phi_0^{-(\kappa-1)} |\log \Phi_0|)$. The total number of 
Laplace transform evaluations is, upon selecting $p = O(\log N)$, $|\log \lambda| = O(N^{-1} \log N)$, and $L = N$  
\begin{align}
M p + N \Phi_0   = O( N^{\frac{\kappa-1}{\kappa}} \log N).  
\end{align}

For high order Radau IIA methods, we thus observe a sublinear complexity regarding the number of Laplace transform evaluations. However, the complexity gain 
is diminished for high order methods as $\kappa$ increases with increasing number of stages $m$.

\section{The sectorial case}\label{sect:sectorial}
For families of operators $K(s)$ that are analytically extendable to a sector larger than a half space, we expect to require significantly 
fewer Laplace transform evaluations. Let us assume a slightly simplified setting, where, for a constant $M>0$ and $\mu\in \mathbb R$, 
\begin{align}\label{Z-bound-sect}
	\norm{K(s)}_{Y\leftarrow X}&\leq M \abs{s}^\mu\quad  \text{ for \ Re }s > -\gamma \abs{\Im s}.
\end{align} 
As before for the implicit Euler in the half-space case,  we motivate the construction of our approximation by tracing the domain of analyticity, which is given by the set 
 \begin{align}\label{eq:cond-sect}
	D^\delta := \{\varphi \in \mathbb C\colon \Re \delta(\lambda e^{i\varphi }) > \Im \delta(\lambda e^{i\varphi })\}.
\end{align}
Our attention is restricted to the implicit Euler method and we intend to follow the argument structure of Remark~\ref{rem:motivation-recursion-euler}, 
to derive a mesh ${\mathcal T}$ without the additional difficulties introduced by the technical treatment conducted in Theorem~\ref{thm:euler}.
With the characteristic function $\delta(\zeta) = 1 - \zeta$ of the implicit Euler in the condition of \eqref{eq:cond-sect} we get  
\begin{align*}
	1-\lambda e^{-\varphi_i} \cos(\varphi_r)> - \gamma\abs{\sin(\varphi_r)}\lambda e^{-\varphi_i}. 
\end{align*}
After rearranging the equation for $\varphi_i$, we arrive at 
\begin{align*}
e^{\varphi_i} > \lambda\left( \cos(\varphi_r)-\gamma \abs{ \sin(\varphi_r) }\right).
\end{align*}
Taking the natural logarithm on both sides yields the desired condition on the imaginary part $\varphi_i$, which reads
\begin{align*}
  \varphi_i > \log(\lambda)+ \log(\cos(\varphi_r) - \gamma \abs{\sin(\varphi_r)} ).
\end{align*}
We note that the right-hand side is generally negative and for small $\varphi_r$, we have the expansion
\begin{align*}
\varphi^{\text{min}}_i = \log(\lambda) - \gamma \varphi_r - \varphi_r^2.
\end{align*}
The situation has now substantially improved, since a term with a lower order dependence on $\varphi_r$ has appeared on the right-hand side. Enforcing a uniform convergence rate for a sequence of intervals, yields again a recursion for the endpoints of the intervals: We set
\begin{align*}
	\dfrac{\abs{\log(\lambda)}+\gamma\Phi_j+\Phi_j^2}{\Phi_{j+1}-\Phi_j} = c .
\end{align*}
Rearranging for $\Phi_{j+1}$ yields 
\begin{align*}
	\Phi_{j+1}= (1+c^{-1}\gamma)\Phi_j+c^{-1}\abs{\log(\lambda)}+c^{-1}\Phi_j^2>
	(1+c^{-1}\gamma)^{j}\Phi_1.
\end{align*}
The magnitude of the interval endpoints therefore increases exponentially in $j$ with a rate that is independent of $\tau$. 
As the starting point, we simply choose $\Phi_0=0$. The total number of elments in ${\mathcal T}$ is there $O(\log N)$. 
On each interval we require again $p = O(\log N)$ Laplace transform evaluations, leading to a total of $O(\log^2 N)$ Laplace transform evaluations. 
\section{Implementation and numerical experiments}\label{sect:numerics}
Let $\Phi_{-1} = 0$ and $\Phi_j$, for $0 \le j \le M$, be the nodes determined by a recursion analogous to \eqref{eq:rec}, with $M$ given by \eqref{eq:M}. We refer to Section~\ref{sec:kernel-approximation} for a detailed description of the mesh.
Moreover, we denote the corresponding set of Lagrange polynomials by $(l^j_k)_{k = 0,\ldots,p}$. The interpolation operator, again restricted to the interval $(\Phi_{j-1},\Phi_j)$, then has the explicit form 
\begin{align*}
 I^{[\Phi_{j-1},\Phi_j]}_{p}\cqK\left(\dfrac{\delta(\lambda \,e^{-i \varphi})}{\tau} \right) 
 = \sum_{k=0}^p \cqK\left(\dfrac{\delta(\lambda \,e^{-i \varphi^j_k})}{\tau} \right) l^j_k(\varphi), \quad \text{for } \varphi\in (\Phi_{j-1},\Phi_j) .
\end{align*}
Note that assembling this form explicitly, at a given point $x\in [0,2\pi]$,  would generally require matrix--matrix operations and is therefore computationally prohibitive. Similarly to practical considerations for the convolution quadrature method in general, we avoid this difficulty by purely relying on matrix--vector and vector--vector operations.
In the following, we discuss practical aspects of the realization of \eqref{eq:forward-cq-interpolation}, namely, the computation of the convolution
\begin{align}\label{eq:approximate_convolution}
	\left(I^{\mathcal T}_p \cqK(\pt^{\tau,\lambda})\cqg\right)_n =	\dfrac{\lambda^{-n}}{L}\sum_{l=0}^{L-1}\zeta_L^{ln}  
	 I^{\mathcal T}_{p}\cqK\left(\dfrac{\delta(\lambda \,e^{-i 2\pi l /L} )}{\tau} \right)	\widehat g^l,
\end{align}
where again $\zeta_L=  e^{2\pi i  /L}$ and $\widehat g^l \in X$ for $0\le l \le L-1$ is given by 
\begin{align*}
	\widehat g^l =  \sum_{j=0}^N \lambda^j  \zeta^{-jl}_{L} \cqg^j.
\end{align*}
\subsubsection*{Evaluation at all time points $(t_n)_{n\le N}$}

We introduce the double index $0=i_0\le \dots \le i_j \le \dots \le i_{K} = 2M+1$ to describe the location of the equidistant points $(y_l)_{l=0}^{L-1}=(2\pi l/L)_{l=0}^{L-1}$ within the interval partitions via 
\begin{align*}
	i_{j-1} \le l \le i_{j} \quad
	\iff 
	\quad y_l \in I_j= [\Phi_{j-1},\Phi_{j}].
\end{align*}
Splitting the sum \eqref{eq:approximate_convolution} into these discrete partitions then yields
\begin{align}\label{eq:compute-all-times}
	\left(I^{\mathcal T}_p \cqK(\pt^{\tau,\lambda})\cqg\right)_n 
	=
	\dfrac{\lambda^{-n}}{L}{\sum_{j=0}^{2M+1}} {}^{\!\! \prime}\,\sum_{l=i_{j-1}}^{i_{j} }\zeta_L^{ln} \left(  
	\sum_{k=0}^pl^{j}_k(y_l)\cqK\left(\dfrac{\delta(\lambda \,e^{-i \varphi_k^j})}{\tau} \right) 
	\widehat{g}^l \right).
\end{align}

Computing $\widehat g^l$ for $0 \le l \le L-1$ and the summation over the index $l$  (which was partitioned via the index $j$) are still realized by the fast discrete fourier method (FFT), as is the case with the standard implementation of the convolution quadrature method. Moreover, $L\cdot (p+1)$ matrix-vector products have to be evaluated.  The prime in the outer $j$-summation denotes that the summands are given by the interpolated expression for $j=1,\dots,2M$ and the summands corresponding to the first and last intervals ($j=0$ and $j=2M+1$) are evaluated through their standard expression \eqref{eq:forward-cq}. In other words, we replace the interpolation $I^{\mathcal T}_p \cqK$ in the first and last interval by the evaluation of the exact operator $\cqK$. In practice, we simply check on each subinterval $I_j$ if the amount of nodes of the global trapezoidal rule exceeds $p+1$ and only then apply the interpolation procedure.  As with the standard convolution quadrature method, only a single evaluation of $K(s)$ has to be stored in memory simultaneously.

\subsubsection*{Evaluation at a single time point $t_n$}
Changing the order of summation in \eqref{eq:compute-all-times} yields a practical realization of one (or a few) time steps $t_n$, where each Laplace domain evaluation is only required to perfom one (or a few, respectively) matrix--vector operations. More precisely, we have
\begin{align*}
	\left(I^{\mathcal T}_p\cqK(\pt^{\tau,\lambda})\cqg\right)_n =	\dfrac{\lambda^{-n}}{L}\sum_{j=0}^{2M+1}{}^{\!\! \prime}\,
	\sum_{k=0}^p\cqK\left(\dfrac{\delta(\lambda \,e^{-i \varphi_k^j})}{\tau} \right) \sum_{l=i_{j-1}}^{i_{j} }\zeta_L^{ln}   l^{j}_k(x_l)
	\widehat{g}^l.
\end{align*}
 The evaluation of this expression requires $(2M+1) \cdot (p+1)$ matrix-vector products, as well as a single FFT evaluation in order to compute $(\widehat g^l)_{l=0,\ldots,L-1}$. 

With this formulation, fast methods for the solution of nonlinear convolution equations can be derived with the classical ideas from \cite{HLS85}.

\subsection{Numerical experiments: A scalar toy example}
The following experiment is taken from \cite[Section 3.5]{banjai-sayas22} and is motivated by the three-dimensional scattering from a sphere with a radially symmetric incoming wave (see \cite[Example~{A.3}]{banjai-sayas22}) and homogeneous Dirichlet boundary conditions. The corresponding analytic function does not extend beyond a half-space and reads 
\begin{align}\label{hyperbolic-kernel}
K(s)=\frac{s}{1-e^{-s}} = s \sum_{j=0}^\infty e^{-js} \quad \text{for } \Re s \ge 0.
\end{align}
Due to the periodically occuring poles along the imaginary axis at $(i2\pi k)_{k\in \mathbb Z}$, this function is not sectorial for any angle.
As the right-hand side, we choose a function that is only differentiable once at $t=0$: 
\begin{align*}
g (t) = \begin{cases}
	0 & t\le 0, \\ 
	\sin(t)^2 & t>0 .
\end{cases}
\end{align*}
Using the geometric series expression of $K(s)$, we obtain
\begin{align}\label{eq:experiment-scalar}
	K(\partial_t)g (t) = \sum_{j=0}^\infty \partial_t g(t-j).
\end{align}
We employ the convolution quadrature method, based on different time stepping schemes to the temporal convolution \eqref{eq:experiment-scalar}. We refer to the approximation $\eqref{eq:approximate_convolution}$ as the \emph{modified} convolution quadrature method. 

As the error norm, we use the weighted temporal $l^2$ norm, i.e., 
\begin{align*}
	\text{err} = \sqrt{\tau \sum_{j=0}^{N} |I^{\mathcal T}_p K(\partial_t^{\tau,\lambda})g_n - K(\partial_t)(t_n) |^2 },
\end{align*}
where the exact convolution is computed by \eqref{eq:experiment-scalar}. 

For all experiments, we use the parameter set of Section~\ref{sect:approx-implic-euler}, with $L=N$. The Laplace domain function $K(s)$ in \eqref{hyperbolic-kernel} fulfills the polynomial bound \eqref{Z-bound} with $\mu = 1$, as long as a numerical shift as in Remark~\ref{rem:simpl-K} is applied. Numerically, such a shift was not necessary in our experiments. If necessary, the formulation can either be slightly rewritten via Remark~\ref{rem:simpl-K}, or the intervals defined via the recursion formula \eqref{eq:rec} can be shortened by a constant factor, to take into account that bounds of the absolute modulus of $K(s)$ deteriorate as $\Re s\rightarrow 0$. For the given data $g$, the standard convolution quadrature approximation based on the implicit Euler method converges with the order $k=1/2$ for the implicit Euler method and with order $k=2/3$ for backward differential formulas of order $p=2$. The exponential convergence rate of the interpolation is set to $\rho = 2$. With these settings, we construct $\epsilon,\lambda$ and $p$ according to \eqref{eq:parameter-choice}, throughout the experiments. 
	
As the recursion for the implicit Euler we use an \eqref{eq:rec}, for the BDF-$2$ method we use an adaption of \eqref{eq:rec-BDF2}, where we introduce a positive factor $\gamma$ and set
	\begin{align*}
		\Phi_{j+1} = \Phi_j + \abs{\log(\lambda)} +\gamma\Phi_j^2 .
\end{align*}
In the experiment we conducted, $\gamma = 3/10$ was an effective choice.

As the first experiment, we use an underlying implicit Euler based convolution quadrature approximation. We approximate the convolution until the final time $T=5$ and use various numbers of time steps $N=5\cdot 2^k$ for $k=2,\dots,15$. Fig.~\ref{fig:BDF1-scalar-HS} depicts the results, where the left plot shows the number of Laplace transform evaluations needed in order to achieve the error given on the $y$-axis. The lack of regularity of $g$ in the origin causes order reduction of the convolution quadrature method, which only converges with the error $O(\sqrt{\tau})$. Since the standard convolution quadrature method requires $L=N$ evaluations of the Laplace transform, the overall method converges in the order $O(1/\sqrt{L})$, where $L$ denotes the number of required Laplace transform evaluations. 

The modified scheme roughly doubles the convergence order, with respect to the number of necessary Laplace transform. This is predicted by Theorem~\ref{thm:euler}, since the number of intervals created by the recursion roughly scale with $O(\sqrt{N})$ and the polynomial degree is of the order $p=O(\log N)$. The plot indicates that the logarithmic terms do not play a significant role for practical computations based on the implicit Euler method. On the right-hand side of Fig.~\ref{fig:BDF1-scalar-HS}, we depict the difference between the convolution quadrature method and the modified scheme. By construction (i.e., by setting the parameters as described in Section~\ref{sect:approx-implic-euler}), this difference is of the same order as the error introduced by the convolution quadrature method itself. 

Fig.~\ref{fig:BDF2-scalar-HS} compares the errors of the BDF2 method, the parsimonious BDF2 method, and the trapezoidal rule applied to the same convolution \eqref{eq:experiment-scalar}. The difference in the number of Laplace transforms required to achieve a certain accuracy is not as prevalent here, although the parsimonious BDF2 method does require the least number of Laplace transforms among the methods compared. The difference of the BDF2 method and its parsimonious counterpart is less robust here, although increasing the polynomial degree slightly increases the accuracy to a satisfactory level. 
\begin{figure}
	\centering
	\includegraphics[scale=0.6]{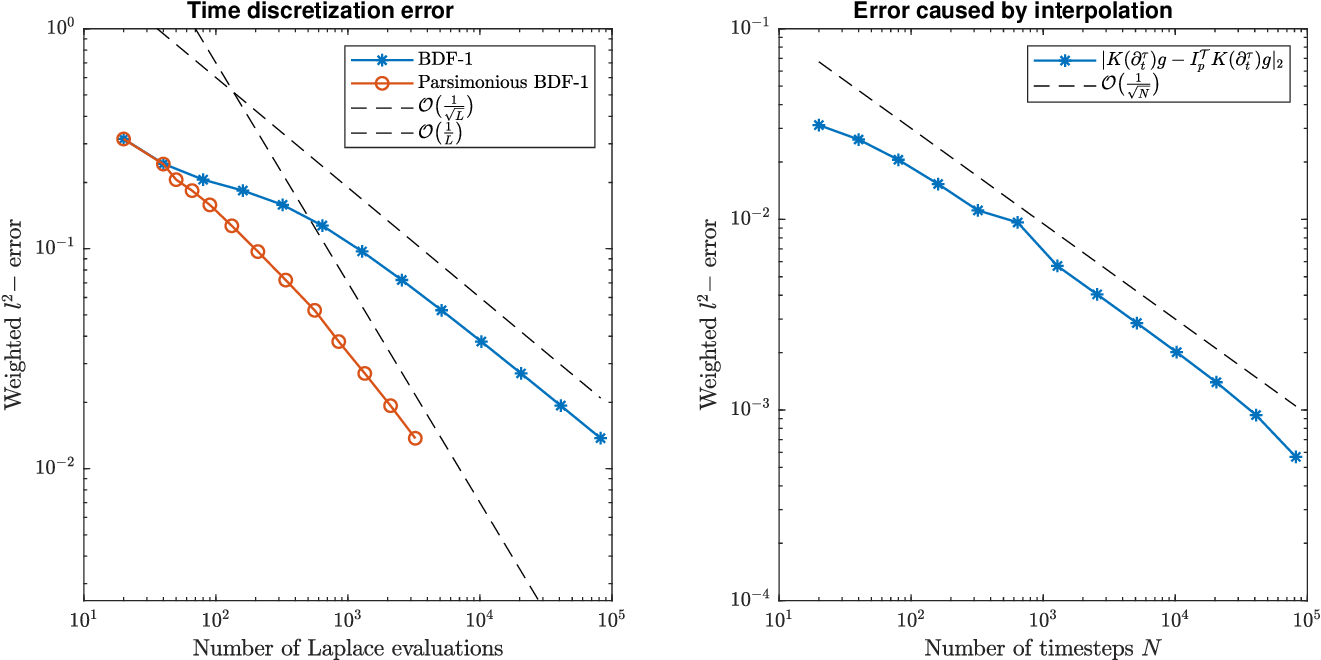}
    \caption{ The convergence of an implicit Euler based convolution quadrature approximation and the corresponding modified method applied to the hyperbolic example \eqref{eq:experiment-scalar}. Shown is the error on the $y$-axis, which is set in relation to the number of Laplace evaluations on the $x$-axis. On the right-hand side, we observe the error introduced by the interpolation used to compute the modified approximation. }
    \label{fig:BDF1-scalar-HS}
\end{figure}
\begin{figure}
	\centering
	\includegraphics[scale=0.6]{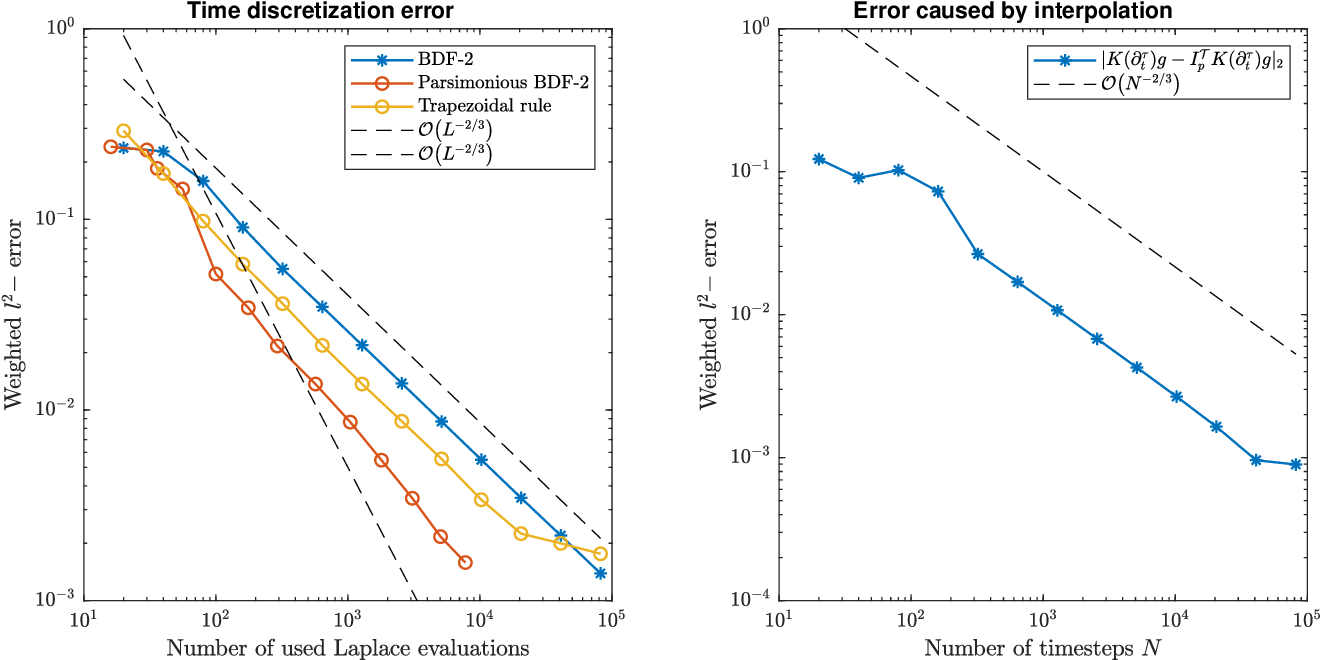}
	\caption{Error vs.\ number of Laplace transform evaluations $L$ for CQ based on different multistep methods applied to the hyperbolic example \eqref{hyperbolic-kernel}. }
    \label{fig:BDF2-scalar-HS}
\end{figure}
\subsubsection*{The sectorial case}
As a sectorial counterpart, we choose the adjusted kernel
\begin{align}\label{sectorial-kernel}
	K_\alpha(s)=\frac{s}{1-e^{-s^{\alpha}}} .
\end{align}
For $\alpha=\frac{2}{3}$, this function is analytic on a sector with an angle of $45$ degrees to the imaginary axis, which corresponds to $\gamma=1$. In order to estimate the errors, we use a reference solution with a implicit Euler based discretization with $N=5\cdot 2^{20}$ time steps. We observe that the modified scheme requires significantly fewer evaluations, however, the slope in the convergence plot is slightly less stable than in Figure~\ref{fig:BDF1-scalar-HS}.
\begin{figure}
	\centering
	\includegraphics[scale=0.6]{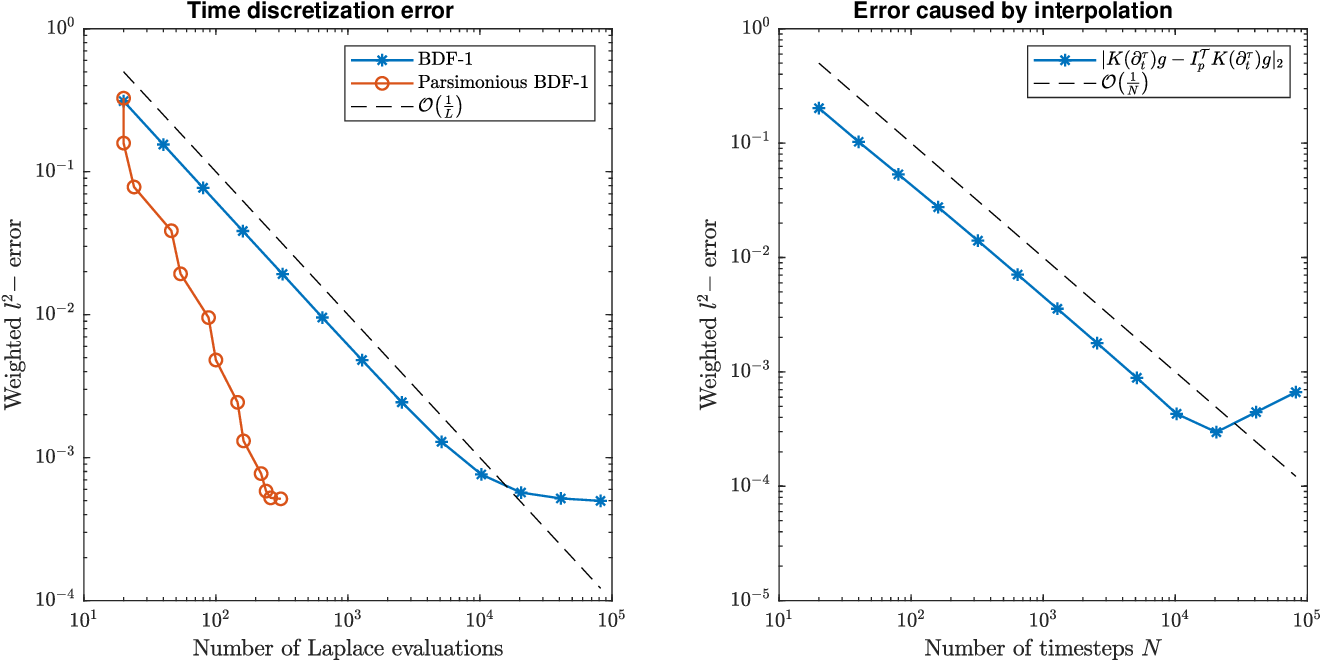}
	\caption{CQ based on implicit Euler applied to the sectorial example \eqref{sectorial-kernel} with $\alpha=\frac{2}{3}$, which corresponds to $\gamma=1$ (i.e., a sector with $45^{\circ}$ angle).
		}
   \label{fig:BDF1-scalar-Sect}
\end{figure}
\begin{figure}
	\centering
	\includegraphics[scale=0.6]{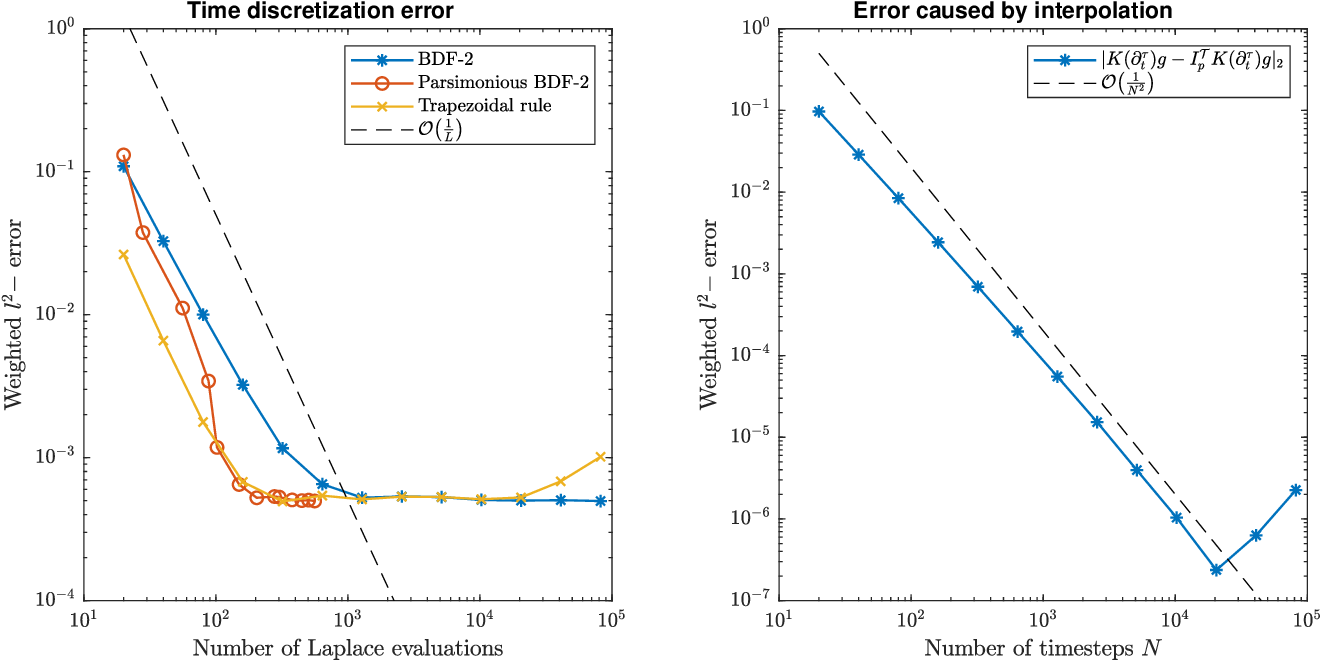}
	\caption{CQ based on different multistep methods of order $2$ applied to the sectorial example \eqref{sectorial-kernel} with $\gamma=1$ (which corresponds to a sector with $45^{\circ}$ angle). 
		}
	\label{fig:BDF2-scalar-Sect}
\end{figure}


\subsection{Numerical experiments: An acoustic scattering problem}\label{sect:num-ac}
Finally, we consider an acoustic wave scattering problem. Let $\Omega = {\mathbb R}^3 \setminus [0,1]^3$ with boundary $\Gamma = \partial \Omega$. Further let $u^{\text{inc}} \, : \, [0,T] \times \mathbb R^3 \rightarrow \mathbb R$ be a solution to the acoustic wave equation in the free space $\mathbb R^3$ with sufficient regularity and initial support away from the boundary $\Gamma$. The acoustic scattering problem then reads
\begin{alignat*}{2}
\partial_t^2 u - \Delta u &= 0 \,\quad &&\text{in} \quad \Omega,
\\
u &= -u^{\text{inc}} \, \quad &&\text{on} \quad\Gamma.
\end{alignat*}
As the incoming wave, we choose the temporal plane wave
\begin{align*}
u^{\text{inc}} 
=
e^{-10(x_1 -tj + 4)^2}.
\end{align*}
The initial support of this expression is not zero at the boundary but negligible. We characterize the scattered wave as the solution to a boundary integral equation. Let $S(s) \, \colon \, H^{-1/2}(\Gamma) \rightarrow H^1(\Omega)$ and $V(s) \, \colon \, H^{-1/2}(\Gamma) \rightarrow  H^{1/2}(\Gamma)$ denote the (Laplace domain) single-layer potential and boundary operators , which are defined by the convolutions 
\begin{align}\label{eq:S and V}
S(s)\varphi (x) &= \int_\Gamma \frac{e^{-s|x-y|}}{4\pi |x-y|} \varphi(y) \, \mathrm d y \ \  \text{for} \ \ x \in \Omega, 
\quad\quad\quad
V(s)\varphi (x) = \int_\Gamma \frac{e^{-s|x-y|}}{4\pi |x-y|} \varphi(y) \, \mathrm d y \quad \text{for} \quad x \in \Gamma.
\end{align}
 The exact solution of this scattering problem is explicitly given by 
\begin{align}\label{eq:experiment-scat}
	u = - [SV^{-1}](\partial_t) u^{\text{inc}}.
\end{align}
This convolution is well defined for sufficiently regular incoming waves, since
$S(s)V^{-1}(s)$ fulfills the bound \eqref{Z-bound} for $\mu=2$. This approach in the time domain was originally described in \cite{BHD86} and was used in combination with the convolution quadrature in \cite{L94}. More details are found in, e.g., \cite[Sec.~{4.6}]{banjai-sayas22}.

The boundary integral operators \eqref{eq:S and V} are realized with the boundary element method library {\tt bempp-cl}, \cite{bemppcl}, where lowest elements are chosen. 
To focus on the error introduced by the time discretization, we fix a grid with $2836$ degrees of freedoms and compare the numerical solution with different time step sizes.

The convolution \eqref{eq:experiment-scat} is approximated until the final time $T=3$, at which we estimate the error of the scattered wave.
 To estimate the error, we compute the scattered wave at different points in the domain $\Omega$. In our experiment, we discretize the cube $[-1.5,3.5]^3$ by an equidistant grid of $1000$ points ($10$ equidistant points in each direction). The approximations at points that are inside the scatterer are disregarded. By using a reference solution with $N= 5120 = 5\cdot 2^{10}$ time steps, we estimate the time discretization error at these points and average the errors. Figure~\ref{fig:BDF1-bempp} shows the results, where we observe that structurally, the behavior of the scalar counterpart depicted in Figure~\ref{fig:BDF1-scalar-HS} carries over. This is to be expected, since the number of necessary Laplace transforms is only slightly dependent on the Laplace transform $K(s)$ (the dependence only enters through the bound $\mu$ in the choice of the parameters as discussed in Section~\ref{sect:approx-implic-euler}).

\section*{Acknowledgment}
JMM gladly acknowledges financial support by Austrian Science Fund (FWF)
throughout the special research program \textit{Taming complexity in PDE systems} (grant SFB F65,
\href{https://doi.org/10.55776/F65}{DOI:10.55776/F65}).

\begin{figure}
	\centering
	\includegraphics[scale=0.6]{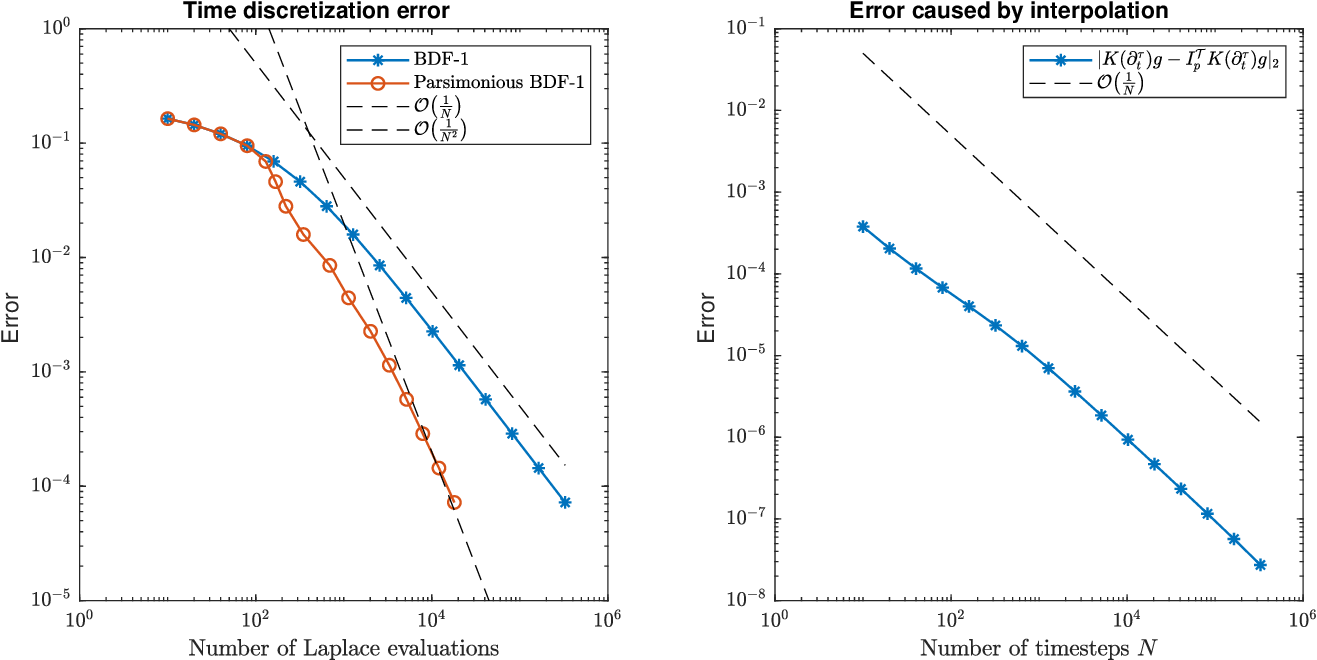}
	\caption{Temporal convergence of CQ based on implicit Euler applied to the convolution \eqref{eq:experiment-scat}, which computes the scattered wave solution to the acoustic scattering problem. The grid has been fixed to a coarse grid with $2836$ degrees of freedom. On the right-hand side, we again observe the error introduced by the modification of the method.}
   \label{fig:BDF1-bempp}
\end{figure}
\bibliographystyle{abbrv}
\bibliography{Lit}
\appendix
\section{Details of the proof of Theorem~\ref{thm:euler}}
\label{sec:details-of-thm:euler}
\begin{numberedproof}{of Theorem~\ref{thm:euler}}

\textbf{Step 0 (2 preparatory estimates):}
\begin{align}
\label{eq:thm:euler-10} 
\forall \Phi \in [\Phi_0,\frac{1}{2}] \quad \colon \quad 
\Phi \ge -\frac{8}{3} \log \left(\cos\frac{7}{8}\Phi\right) - \frac{8}{3} \log \lambda . 
\end{align}
The function $g(\Phi):= \Phi + \frac{8}{3} \log (\cos \frac{7}{8}\Phi)$ is concave so that $\min_{\Phi \in [\Phi_0,1/2]}  g(\Phi) 
 = \min\{ g(\Phi_0), g(1/2)\} \ge \min\{g(\Phi_0), 0.236\} \ge \min\{g(\Phi_0), -\frac{8}{3} \log\lambda\}$ in view of $0.08 \ge - \log \lambda > 0$. 
To see $g(\Phi_0) \ge -\frac{8}{3} \log \lambda$, we write $\Phi_0  = \sqrt{- \frac{5}{4} \frac{8}{3} \log \lambda } \leq \sqrt{\frac{0.8}{3}} \leq 0.52$
and have to show that $g(\Phi_0) + \frac{8}{3} \log \lambda = g(\Phi_0) - \frac{4}{5} \Phi_0^2 \ge 0$. The function 
$\Phi \mapsto g(\Phi) - \frac{4}{5} \Phi^2$ is concave so that $g(\Phi_0) - \frac{4}{5} \Phi_0^2\ge 
\min \{ g(0), g(0.52) - \frac{4}{5} 0.52^2\} = 0$. This shows \eqref{eq:thm:euler-10}. 

From the concavity of $\Phi \mapsto \tilde g(\Phi) := \frac{1}{2} \Phi^2 + \log \cos \Phi$ and $\tilde{g}^\prime(0) = 0$ 
we get $\tilde g^\prime(\Phi)\leq 0$ for $\Phi \in [0,1]$ and therefore $\tilde g(\Phi) \leq 0$ for $\Phi \in [0,1]$ so that 
\begin{align}
\label{eq:thm:euler-20} 
\forall \Phi \in [0,1] \quad \colon \quad 
- \log \cos (\frac{7}{8} \Phi) \ge \frac{1}{2} \frac{49}{64}\Phi^2. 
\end{align}
\textbf{Proof of \ref{item:thm:euler-i}:}
\textbf{Step 1 (Bernstein ellipses are contained in the right half-plane):}
In order to show the interpolation error bound in Step~2 below, we have to ensure that
	$$
\delta(\lambda( E_\rho(\Phi_j,\Phi_{j+1}))) \subset \mathbb C_+ = \{s\in \mathbb C \, | \, \Re s>0 \}
$$
with $\delta(z) = 1-z$ for the implicit Euler method. 
	Inserting the parametrization of the Bernstein ellipsoid on the interval $[\Phi_j,\Phi_{j+1}]$ shows that this property is equivalent to
	\begin{align}\label{est:delta-larger-zero}
		\Re \delta\left(\lambda \exp\left( i\Phi_j +(\Phi_{j+1} - \Phi_j)\frac{i}{2}\left(\dfrac{\rho e^{i\theta} + \rho^{-1}e^{-i\theta}}{2} +1\right) \right)\right) > 0 ,	
	\end{align}
	for all $\theta\in [0,2\pi]$. In order to separate the real and imaginary parts of the argument in the exponential function, we use Euler's formula to obtain
	\begin{align*}
		i(\rho e^{i\theta} + \rho^{-1}e^{-i\theta})
		&=
		i\cos(\theta)(\rho+\rho^{-1}) +\sin(\theta)(\rho^{-1}-\rho) .
	\end{align*} 
	Inserting this identity in the argument of the expression on the left-hand side of \eqref{est:delta-larger-zero} yields, again in combination 
        with Euler's formula, the identity
	\begin{align*}
		&\Re \exp\left( i\Phi_j +(\Phi_{j+1} - \Phi_j)\frac{i}{2}\left(\dfrac{\rho e^{i\theta} + \rho^{-1}e^{-i\theta}}{2} +1\right) \right)
		\\=
		&  \exp\left((\Phi_{j+1} - \Phi_j)\frac{\rho^{-1}-\rho}{4}\sin(\theta) \right)
		\cos\left(\Phi_j + (\Phi_{j+1} - \Phi_j)\left(\frac{\rho+\rho^{-1}}{4}\cos(\theta) + \frac{1}{2} \right)\right)
		.	
	\end{align*}
	Inserting $\rho=2$, simplifies the expression to 
	\begin{align}
\label{eq:thm:euler-5} 
		&  \exp\left(-(\Phi_{j+1} - \Phi_j)\frac{3}{8}\sin(\theta) \right)
		\cos\left(\Phi_j + (\Phi_{j+1} - \Phi_j)\left(\frac{5}{8}\cos(\theta) + \frac{1}{2} \right)\right)
		.	
	\end{align}
We now proceed to show (\ref{est:delta-larger-zero}) by distinguishing 
several cases for the interval $[\Phi_j,\Phi_{j+1}]$.

\emph{Proof of (\ref{est:delta-larger-zero}) for $1 \ge \Phi_j \ge  -\frac{8}{3} \log \lambda - \frac{8}{3} \log \cos (\frac{7}{8}\Phi_j)\}$:}  
We estimate 
	\begin{align*}
		&1- \Re \lambda \exp\left( i\Phi_j +(\Phi_{j+1} - \Phi_j)\frac{i}{2}\left(\dfrac{\rho e^{i\theta} + \rho^{-1}e^{-i\theta}}{2} +1\right) \right)
		\\ \stackrel{-\sin \theta \leq 1} {\ge} &
		1 - \lambda\exp\left((\Phi_{j+1} - \Phi_j)\frac{3}{8}\right)
		\cos\Bigl(\Phi_j + \underbrace{(\Phi_{j+1} - \Phi_j)}_{\leq \Phi_j} \underbrace{\left(\frac{5}{8}\cos(\theta) + \frac{1}{2} \right)}_{ \in [-\frac{1}{8}, \frac{9}{8}]}\Bigr)
		\\ > & 
		1 - \lambda\exp\left((\Phi_{j+1} - \Phi_j)\frac{3}{8}\right)
		\cos\left(\frac{7}{8}\Phi_j\right) = 0, 
	\end{align*}
	where the final equality holds by the construction of the recursion \eqref{eq:rec}.

\emph{Proof of (\ref{est:delta-larger-zero}) for $\Phi_j <   -\frac{8}{3} \log \lambda - \frac{8}{3} \log \cos (\frac{7}{8}\Phi_j)\}$  together with $\Phi_j \leq 1$:} 
These assumptions imply
by Step~0 that $\Phi_j \in [1/2,1]$. We estimate 
	\begin{align*}
		&1- \Re \lambda \exp\left( i\Phi_j +(\Phi_{j+1} - \Phi_j)\frac{i}{2}\left(\dfrac{\rho e^{i\theta} + \rho^{-1}e^{-i\theta}}{2} +1\right) \right) \\
 & \qquad \ge 1- \lambda \exp(-\frac{3}{8} \Phi_j \sin\theta) \cos(\Phi_j (\frac{3}{2} + \frac{5}{8}\cos\theta)) 
=: F(\Phi_j,\theta). 
\end{align*}
Using that $\Phi_j \in [1/2,1]$, one can assert
that $\partial_1 F(\Phi_j,\theta) > 0$ for any $\theta$ in the following way: 
$\partial_1 F(\Phi,\theta) > 0$ if $G:=\frac{3}{8} \sin\theta \cos (\Phi(\frac{3}{2} + \frac{5}{8} \cos\theta)) + (\frac{3}{2} + \frac{5}{8} \cos\theta)\sin(\Phi(\frac{3}{2} + \frac{5}{8} \cos\theta)) > 0$. To see this, one observes $\frac{3}{2} + \frac{5}{8} \cos \theta \in [\frac{7}{8}, \frac{17}{8}]$ 
so that for $\Phi \in [\frac{1}{2},1]$ one has 
$(\frac{3}{2} + \frac{5}{8} \cos\theta)\sin(\Phi(\frac{3}{2} + \frac{5}{8} \cos\theta)) 
\ge \frac{7}{8} \min\{\sin(\frac{1}{2} \frac{7}{8}), \sin(\frac{1}{2}\frac{17}{8}\} \ge \frac{7}{8} \sin \frac{7}{16}$.  
Finally, $ G \ge -\frac{3}{8} + \sin \frac{7}{16} \approx 0.048 >  0$.  

The assertion $\partial_1 F(\Phi_j,\theta) > 0$ allows us conclude that 
$\min_{\theta \in [0,2\pi]} F(\Phi_j,\theta) \ge 0$, if $\min_{\theta \in [0,2\pi]} F(1/2,\theta) \ge 0$. The latter follows from
graphical considerations using Mathematica. 

\emph{Proof of (\ref{est:delta-larger-zero}) for $\Phi_j \ge 1$:} 
	For $\Phi_j \ge 1$, the recursion terminates at the next step, i.e., we have $\Phi_{j+1}=\pi$. We can estimate
	\begin{align*}
		& \min_{\theta \in[0,2\pi]} 1- \lambda\exp\left(-(\Phi_{j+1} - \Phi_j)\frac{3}{8}\sin(\theta) \right)
		\cos\left(\Phi_j + (\Phi_{j+1} - \Phi_j)\left(\frac{5}{8}\cos(\theta) + \frac{1}{2} \right)\right)
	\\	
& \qquad \ge 
	 \min_{\theta \in[0,2\pi], \Phi \in [1,\pi]} 1- \exp\left(-(\pi - \Phi)\frac{3}{8}\sin(\theta) \right)
	\cos\left(\Phi + (\pi - \Phi)\left(\frac{5}{8}\cos(\theta) + \frac{1}{2} \right)\right) =:\widetilde{F}(\Phi,\theta) . 
	\end{align*}
We next check that $\nabla \widetilde{F} = 0$ in $(1,\pi) \times [2,2\pi]$ so that $\widetilde{F}$ attains its minium 
for $\Phi \in \{1,\pi\}$. 
Abbreviating $x:= \Phi + (\pi - \Phi) (\frac{5}{8} \cos\theta + \frac{1}{2})$, one compute
\begin{align*}
\partial_\Phi \widetilde{F} & = 
-\lambda e^{-(\pi - \Phi) \frac{3}{8} \sin \theta} \left[ 
\frac{3}{8} \sin\theta \cos x - \sin x (\frac{1}{2} - \frac{5}{8} \cos \theta)
\right], \\
\partial_\theta \widetilde{F} & = 
-\lambda e^{-(\pi - \Phi) \frac{3}{8} \sin \theta} (\pi - \Phi)\left[ 
-\frac{3}{8} \cos\theta \cos x + \sin x \sin \theta
\right] 
\end{align*}
Unless $\cos x = 0$ or $\cos \theta=0$, the condition $\nabla \widetilde{F} = 0$ implies 
\begin{align*}
\frac{5}{8} \sin^2 \theta = (\frac{1}{2} - \frac{5}{8} \cos \theta) \cos \theta. 
\end{align*}
Graphical considerations show that this is not possible. Hence, $\nabla \widetilde{F} \ne 0$ on $(1,\pi) \times [2,\pi]$ and 
thus $\widetilde{F}$ attains its minimum for $\Phi \in \{1,\pi\}$. 
For $\Phi = \pi$ one has $\widetilde{F} > 0$ and for $\Phi = 1$, 
	graphical considerations with Mathematica show $\min_{\theta \in [0,2 \pi]}\widetilde{F}(1,\theta)  > 0$. 
In conclusion, $\widetilde{F} > 0$. 

\textbf{Step 2 (Chebyshev interpolation error estimate):}
Finally, for the ellipse $D_\rho(\Phi_j,\Phi_{j+1}) \subset \mathbb{C}$ enclosing by $E_\rho(\Phi_{j},\Phi_{j+1})$, 
we obtain from Step~1 and the assumption \eqref{Z-bound}
\begin{align*}
\sup_{\Phi \in D_\rho(\Phi_j,\Phi_{j+1})} \left| K\left( \dfrac{\delta( \lambda e^{i\Phi }) }{\tau}\right) \right| \le \frac{CM_0}{\tau^\mu},
\end{align*}
where $C$ is a constant independent of $K$, $\lambda$, $\tau$ and $\Phi$. The approximation result \eqref{eq:thm:euler-1} follows from 
combining this with the best approximation result Lemma~\ref{lemma:polynomial-approximation} as formulated 
in (\ref{eq:chebyshev-interpolation-error-scaled}). 
%
%
\end{numberedproof}
\end{document}